\documentclass[letterpaper]{article}
\usepackage{amsthm}
\usepackage{amssymb,latexsym,graphics,enumerate}
\usepackage[mathscr]{eucal}
\usepackage{amsmath,amsfonts,amsthm,amssymb}
\usepackage{color}
\usepackage{hyperref}
\usepackage[left=1in,right=1in,top=1in,bottom=1in]{geometry}

\numberwithin{equation}{section}

\newcommand{\rr}{\mathbb{R}}

\newcommand{\ga}{\Gamma}
\newcommand{\lan}{\langle}
\newcommand{\ran}{\rangle}
\newcommand{\be}{\begin{eqnarray*}}
\newcommand{\bel}{\begin{eqnarray}}
\newcommand{\ee}{\end{eqnarray*}}
\newcommand{\eel}{\end{eqnarray}}
\newcommand{\ba}{\begin{aligned}}
\newcommand{\ea}{\end{aligned}}
\newcommand{\de}{\Delta}
\newcommand{\al}{\alpha}
\newcommand{\na}{\nabla}
\newcommand{\ep}{\epsilon}

\newcommand{\pa}{\partial}
\newcommand{\wh}{\widehat}
\newcommand{\vs}{\vspace{5 mm}}
\newcommand{\pay}{\partial_{y_1}}
\newcommand{\payy}{\partial_{y_2}}
\newcommand{\CC}{C_{2,\infty}}

\newtheorem{thm}{Theorem}
\newtheorem{theorem}{Theorem}

\newtheorem{lem}{Lemma}
\newtheorem{pro}{Proposition}
\newtheorem{rmk}{Remark}
\newtheorem{remark}{Remark}

\numberwithin{remark}{section}
\numberwithin{lem}{section}

\newcommand{\norm}[1]{\left\lVert#1\right\rVert}
\newcommand{\abs}[1]{\left\vert#1\right\vert}
\newcommand{\set}[1]{\left\{#1\right\}}
\newcommand{\brak}[1]{\langle #1 \rangle}
\newcommand{\grad}{\nabla}

\newcommand\Torus{{\mathbb T}}
\newcommand\Real{{\mathbb R}}
\newcommand\Integers{{\mathbb Z}}
\newcommand{\dss}{\displaystyle}
\newcommand{\cF}{\mathcal{F}}

\title{Suppression of blow-up in Patlak-Keller-Segel via shear flows}
\date{\today}
\author{Jacob  Bedrossian\footnote{\textit{jacob@cscamm.umd.edu}, University of Maryland, College Park.
Partially funded by NSF grant DMS-1462029 and an Alfred P. Sloan Research Fellowship}, \, Siming He \footnote{\textit{simhe@math.umd.edu}, University of Maryland, College Park. Funded by Patrick and Marguerite Sung Fellowship and ONR grant N000141512094.}}
\begin{document}
\maketitle%\tableofcontents

\begin{abstract}
In this paper we consider the parabolic-elliptic Patlak-Keller-Segel models in $\Torus^d$ with $d=2,3$ with the additional effect of advection by a large shear flow.
Without the shear flow, the model is $L^1$ critical in two dimensions with critical mass $8\pi$; solutions with mass less than $8\pi$ are global and solutions with mass larger than $8 \pi$ with finite second moment, all blow up in finite time. In three dimensions, the model is $L^{3/2}$ critical and $L^1$ supercritical; there exists solutions with arbitrarily small mass which blow up in finite time arbitrarily fast.
We show that the additional shear flow, if it is chosen sufficiently large, suppresses one dimension of the dynamics and hence can suppress blow-up.
In two dimensions, the problem becomes effectively $L^1$ subcritical and so all solutions are global in time (if the shear flow is chosen large).
In three dimensions, the problem is effectively $L^1$ critical, and solutions with mass less than $8\pi$ are global in time (and for all mass larger than $8\pi$, there exists solutions which blow up in finite time).
\end{abstract}

\setcounter{tocdepth}{1}
{\small\tableofcontents}

\section{Introduction}
Consider the parabolic-elliptic Patlak-Keller-Segel model in $\Torus^d$ with the additional effect of a large shear flow
\begin{equation}\label{def:PKS}
  \left\{
\begin{array}{l} \dss
\partial_t n + Au\partial_x n + \grad \cdot (n \grad c) = \Delta n \\
-\Delta c = n - \bar{n} \\
n(t=0,x,y) = n_{\mbox{{\scriptsize in}}}(x,y),
\end{array}
\right.
\end{equation}
where $\bar{n}$ denotes the average of $n$. If $d = 3$, then we denote $y = (y_1,y_2)$. 
Here, $u = u(y)$ if $d=2$ and $u=u(y_1)$ if $d=3$, is a fixed, $C^3$ function with at most finitely many non-degenerate critical points.
In the case $A = 0$, this system is one of the fundamental models for the study of aggregation via chemotaxis of certain microorganisms; see e.g. \cite{Patlak,KS,Hortsmann,HandP}.
The quantity $n$ denotes the density of microorganisms, which are executing a random walk with a bias up the gradient of the \emph{chemo-attractant} $c$.
 The second equation describes the quasi-static equilibriation and production of the chemo-attractant by the microorganisms.
Patlak-Keller-Segel and its variations have received considerable mathematical attention over the years, for example, see the review \cite{Hortsmann} or some of the representative works \cite{ChildressPercus81,JagerLuckhaus92,Nagai95,Biler95,HerreroVelazquez96,Biler06,BlanchetEJDE06,Blanchet08,BlanchetCalvezCarrillo08,BlanchetCarlenCarrillo10} and the references therein.
The case $A \neq 0$ models the microorganisms suspended in a shear flow: 
the elliptic equation $-\Delta c = n - \bar{n}$ arises as the formal limit as $\epsilon \rightarrow 0$ of the advection-diffusion equation
\begin{align*}
\partial_t c + A \partial_x c = \epsilon^{-1}\left(\Delta c + n\right),
\end{align*}
under the assumption that $\epsilon A \ll 1$. In particular, \eqref{def:PKS} requires that the time-scale of equilibriation of $c$ is faster than the transport due to the shear flow.

One of the most well-known features of \eqref{def:PKS} is that it is $L^1$ critical in two dimensions and is $L^{3/2}$ critical in three dimensions.
For any reasonable notion of solution, 
 the $L^1$ norm of the density is conserved, $M := \norm{n(t)}_{L^1} = \norm{n_{in}}_{L^1}$ (for \eqref{def:PKS}, this is the \emph{mass}).
There is also the \emph{free energy}, for which \eqref{def:PKS} (with $A = 0$) is formally a gradient flow with respect to the $L^2$ Wasserstein metric:
\begin{align}
\mathcal{F}[n(t)] = \int n \log n dx - \frac{1}{2}\int \abs{\grad c}^2 dx \leq \mathcal{F}[n_{in}]. \label{def:F}
\end{align}
In $\Real^2$, the conservation of mass and dissipation of the free energy (the latter is a logarithmically subcritical quantity) can be used to prove that if
$\norm{n_{in}}_{L^1} \leq 8\pi$, then solutions exist for all time (see e.g. \cite{BlanchetEJDE06,Blanchet08,BlanchetCarlenCarrillo10}).
All solutions with finite second moment and $\norm{n_{in}}_{L^1} > 8\pi$ are known to blow up in finite time \cite{JagerLuckhaus92,Nagai95,BlanchetEJDE06}.
On $\Torus^2$, some similar results are known \cite{Senba02}.
On $\Real^3$, since \eqref{def:PKS} is super-critical with respect to the controlled quantities, significantly less is understood.
Solutions which are initially small in $L^{3/2}$ are known to exist globally and it is known that there exists blow-up solutions with arbitrarily small mass \cite{Corrias04}.

In \cite{KiselevXu15} it was shown that if, instead of a shear flow, one has $A\mathbf{u} \cdot \grad n$ where $\mathbf{u}$ is \emph{relaxation-enhancing} -- a generalization of weakly mixing introduced in \cite{CKRZ08} -- then for each smooth initial datum, one can choose $A$ large enough so that the solution to \eqref{def:PKS} does not blow-up in finite time.
Such velocity fields are very good mixers, and this ensures that any non-constant density configuration undergoes a large growth of gradients, and hence a large dissipation. The effect at work is then an \emph{enhanced dissipation}.
This effect has been studied previously in a variety of contexts, such as \cite{CKRZ08,Zlatos2010,BeckWayne11,VukadinovicEtAl2015,BCZGH15,BCZ15}, in the physics literature \cite{Lundgren82,RhinesYoung83,DubrulleNazarenko94,LatiniBernoff01,BernoffLingevitch94}, and in control theory \cite{BZ09,Beauchard2014}; a closely related effect was also studied in \cite{GallagherGallayNier2009}.

Mixing due to a shear flow is quite different from that due to a relaxation-enhancing or weakly mixing flow. In particular, data which is independent of $x$ does not mix at all, and so one must separate the evolution of the zero (or low if $x \in \Real$) frequencies in $x$ from the non-zero frequencies, which is the decomposition into the nullspace of the transport operator and its orthogonal complement.
Enhanced dissipation due to shear flow was shown in \cite{BMV14,BGM15I,BGM15II,BGM15III,BVW16} to be important for understanding the stability of the Couette flow in the 2D and 3D Navier-Stokes equations at high Reynolds number. For example, \cite{BGM15I,BGM15II,BGM15III} show that the enhanced dissipation can suppress 3D effects and simplify the dynamics to be essentially 2D.
It is intuitive then to expect that a large shear flow can also in some sense suppress one dimension in \eqref{def:PKS} and hence make 2D $L^1$ subcritical and 3D $L^1$ critical.
This is essentially what we prove for $u \in C^3$ with finitely many non-degenerate critical points (the relevance of these hypotheses are discussed after the statements).

\begin{theorem} \label{thm:2D}
Let $u \in C^3(\Torus)$ have finitely many, non-degenerate critical points and let $n_{in} \in H^1(\Torus^2) \cap L^\infty(\Torus^2)$ be arbitrary. There exists an $A_0 = A_0(u,\norm{n_{in}}_{H^1},\norm{n_{in}}_{L^\infty})$ such that if $A > A_0$ then the solution to \eqref{def:PKS} is global in time.
\end{theorem}

\begin{remark}
Theorem \ref{thm:2D} extends to the cylindrical domain $\Torus \times \Real$ provided $u'$ is bounded uniformly away from zero near $y \rightarrow \pm \infty$.
\end{remark}

It is clear that Theorem \ref{thm:2D} cannot hold in 3D. Indeed, consider any solution to the 3D problem which is constant in the $x$ direction: $n(t,x,y_1,y_2) = n(t,y_1,y_2)$.
This solution will solve \eqref{def:PKS} on $\Torus^2$ with $A = 0$ and hence the $8\pi$ critical mass will still apply.
Our next result shows that for $A$ large the third dimension is suppressed and $8\pi$ is indeed the critical mass for \eqref{def:PKS} in $\Torus\times \rr^2$ and $\Torus^3$. 
As this setting is effectively critical, Theorem \ref{thm:3D} is harder to prove than Theorem \ref{thm:2D} (which is effectively subcritical, as \cite{KiselevXu15}).

\begin{theorem} \label{thm:3D}
\begin{itemize}
\item[(a)] Let $u \in C^3(\Torus)$ have finitely many, non-degenerate critical points and let $n_{\mbox{{\scriptsize in}}} \in H^1(\Torus^3) \cap L^\infty(\Torus^3)$ be arbitrary such that
$\norm{n_{\mbox{{\scriptsize in}}}}_{L^1} < 8\pi$ and for some $q > 0$, there holds $n_{\mbox{{\scriptsize in}}}(x) \geq q>0$ for all $x \in \Torus^3$.
Then there exists an $A_0 = A_0(u,\norm{n_{\mbox{{\scriptsize in}}}}_{H^1},\norm{n_{\mbox{{\scriptsize in}}}}_{L^\infty},\norm{n_{in}}_{L^1},q)$ such that if $A > A_0$ then the solution to \eqref{def:PKS} is global in time.
\item[(b)] Suppose $u \in C^3(\rr)$ have finitely many, non-degenerate critical points and $u'$ is bounded uniformly away from zero near infinity. Let $n_{in} \in H^1(\Torus\times \rr^2) \cap L^\infty(\Torus\times \rr^2)$ be arbitrary such that $\norm{n_{in}}_{L^1} < 8\pi$ and $I[n_{in}] := \int n_{in}(x,y) \abs{y}^2 dx dy < \infty$.
Consider the problem 
\bel\label{KSC3intro}
\left\{\begin{array}{rrr}\pa_t n+ Au(y_1)\pa_x n+ \na\cdot(\na c n)= \de n,\\
-\de c=n,\\
n(\cdot,0)=n_0.\end{array}\right.
\eel
Then, there exists an $A_0 =  A_0( u, \norm{n_{in}}_{L^\infty},\norm{n_{\mbox{{\scriptsize in}}}}_{H^1},\norm{n_{\mbox{{\scriptsize in}}}}_{L^1},I[n_{in}])$, such that if $A > A_0$ then the solution is global in time.
\end{itemize}
\end{theorem}

\begin{remark}
It is not clear whether or not one could expect Theorem \ref{thm:3D} to hold also in the case $\norm{n_{\mbox{{\scriptsize in}}}}_{L^1} = 8\pi$ as in $\Real^2$ \cite{Blanchet08}.
\end{remark}

Let us now briefly discuss the proofs of Theorems \ref{thm:2D} and \ref{thm:3D}.
By re-scaling time $t \mapsto A^{-1} t$, the system \eqref{def:PKS} is equivalent to
\begin{equation}\label{def:PKSres}
  \left\{
\begin{array}{l} \dss
\partial_t n + u\partial_x n + \frac{1}{A}\left(\grad \cdot (n \grad c) - \Delta n\right) = 0 \\
-\Delta c = n - \bar{n} \\
n(t=0,x,y) = n_{\mbox{{\scriptsize in}}}(x,y),
\end{array}
\right.
\end{equation}
For our purposes, it is convenient to use the form \eqref{def:PKSres}.
In \cite{BCZ15}, enhanced dissipation was studied for the passive scalar equation
\begin{align}
\partial_t f + u\partial_x f = \frac{1}{A}\Delta f. \label{def:passcale}
\end{align}
Among other things, it was shown in \cite{BCZ15} that for $u$ satisfying the hypotheses of Theorems \ref{thm:2D} and \ref{thm:3D}, there exists some $\tilde{\epsilon} > 0$ such that
\begin{align*}
\norm{f(t) - \frac{1}{2\pi}\int_{\Torus} f(t,x,\cdot) dx}_{L^2} \lesssim e^{-\frac{\tilde{\epsilon}\nu^{1/2}}{1 + \abs{\log \nu}^2}t}\norm{f(0)}_{L^2}.
\end{align*}
The technique employed in \cite{BCZ15} is an energy method known as hypocoercivity, see e.g. the text \cite{villani2009} for an overview or \cite{DesvillettesVillani01,DMS15,HerauNier2004,Herau2007,GallagherGallayNier2009} and the references therein.
In the proof of Theorem \ref{thm:2D} we will couple such hypocoercivity energy estimates to $H^1$ energy estimates for the zero-in-$x$ frequency as well as to $L^p$ estimates on \eqref{def:PKS}, similar to the estimates in \cite{JagerLuckhaus92,BlanchetEJDE06,CalvezCarrillo06}, which do not see the advection term.
In the proof of Theorem \ref{thm:3D}, the $x$-independent system is now formally $L^1$ critical, and hence in order to get results for mass up to $8\pi$, we need to employ the free energy in a manner similar to \cite{BlanchetEJDE06}.
However, the \emph{two-dimensional} free energy is not a monotonically dissipated quantity for \eqref{def:PKS}, and hence we need to also couple an estimate on the 2D free energy to the other energy estimates we make and control the errors using the enhanced dissipation. 
This is particularly tricky if one is interested in the result on $\Torus \times \Real^2$.
Enhanced dissipation (or something similar) was studied via hypocoercivity also in \cite{GallagherGallayNier2009,BeckWayne11,BCZ15}, however, to the authors' knowledge, this is the first work that uses hypocoercivity to obtain enhanced dissipation estimates for nonlinear problems.
We remark that the Fourier analysis methods used in \cite{BGM15III,BVW16} also apply to \eqref{def:PKS} in the specific case $u(y) = y$ and $y \in \Real$.
This approach is much simpler than the hypocoercivity methods we employ, however, the hypocoercivity methods allow us to study a much wider variety of shear flows.

\subsection{Notations}
\subsubsection{Miscellaneous}
The constants $B$ below are universal constants which have no dependence on any quantities, except perhaps $u$ and $M$.
 On the contrary, the dependence of the constants $C_{...}$ on various quantities involving $n_{in}$ is more important and will be made a little more explicit. 
Given quantities $X,Y$, if there exists a constant $B$ such that $X \leq BY$, we often write $X\lesssim Y$. 
We will moreover use the notation $\lan x\ran:=(1+|x|^2)^{1/2}$.

\subsubsection{Fourier Analysis}
Most of the time, we consider the Fourier transform \emph{only in the $x$ variable}, and denoting it and its inverse as
\begin{equation*}
\wh{f}_k(y) := \frac{1}{2\pi}\int_{-\pi}^\pi e^{-ikx} f(x,y) dx , \quad \check{g}(x,y) = \sum_{k=-\infty}^\infty g_k(y) e^{ikx}.  
\end{equation*}
Define the following orthogonal projections:
\begin{align*}
f_0(t,y) &= \frac{1}{2\pi}\int_{-\pi}^\pi f(t,x,y) dx ,\\
f_{\neq}(t,x,y)& = f(t,x,y) - f_0(t,y),
\end{align*}
for ``zero frequency'' and ``non-zero frequency''.
For any measurable function $m(\xi)$, we define the Fourier multiplier $m(\grad)f := (m(\xi)\hat{f}(\xi))^{\vee}$.
\subsubsection{Functional spaces}
The norm for the $L^p$ space is denoted as $||\cdot||_p$ or $||\cdot||_{L^p(\cdot)}$:
\begin{align*}
||f||_p=||f||_{L^p}=(\int |f|^p dx)^{1/p},
\end{align*} with natural adjustment when $p$ is $\infty$.
If we need to emphasize the ambient space, we use the second notation, i.e., $||n_{\neq}||_{L^p(\Torus\times\rr^2)}$. Otherwise, we use the first notation for the sake of simplicity. The Sobolev norm $||\cdot||_{H^s}$ is defined as follow:
\begin{align*}
||f||_{H^s}:=||\lan \grad \ran^s f||_{L^2}.
\end{align*}
For a function of space and time $f = f(t, x)$, we use the following space-time norms:
\begin{align*}
||f||_{L_t^pL_x^q}:=&||||f||_{L_x^q}||_{L_t^p},\\
||f||_{L_t^p H_x^s}:=&||||f||_{H^s_x}||_{L_t^p}.
\end{align*}

\section{Proof of Theorem \ref{thm:2D}} \label{sec:2D}
\subsection{Outline of the proof}
In this section, we prove Theorem \ref{thm:2D}.
The enhanced dissipation does not act in the nullspace of the advection term, and hence it is reasonable to decompose the solution as follows
\bel\label{zeromode}\ba
\pa_t n_0+\frac{1}{A}\pa_y(\pa_y c_0 n_0)+\frac{1}{A}(\na\cdot(\na c_{\neq} n_{\neq}))_0=\frac{1}{A}\pa_{yy} n_0,\\
-\de c_0=n_0-\overline{n};
\ea\eel
and,
\bel\label{nonzeromode}\ba
\pa_t n_{\neq} +u(y)\pa_x n_{\neq} -\frac{1}{A}(n_0-\overline{n})n_{\neq}&-\frac{1}{A}n_0n_{\neq}+\frac{1}{A}\na c_0\cdot \na n_{\neq}+\frac{1}{A}\na c_{\neq}\cdot \na n_0 \\
&=-\frac{1}{A}(\na\cdot(\na c_{\neq} n_{\neq}))_{\neq} + \frac{1}{A}\de n_{\neq},\\
-\de c_{\neq}&=n_{\neq}.
\ea\eel
As in \cite{BCZ15}, it is convenient to consider \eqref{nonzeromode} after applying the Fourier transform \emph{only in $x$}.
Applying to both sides of (\ref{nonzeromode}) we have,
\bel\label{nonzeromodeF}\ba
\pa_t \widehat{n}_{k}+NL_k+L_k+u(y)ik \widehat{n}_{k}=&\frac{1}{A}(\pa_{yy}-|k|^2) \widehat{n_{k}},\\
-(\pa_{yy}-|k|^2) c_{k}=&n_{k},
\ea\eel
where $L_k,NL_k$ are defined as follows:
\begin{subequations}
\begin{align}
NL_k:=&-\frac{1}{A}\sum_{\ell \neq 0} \widehat{n}_{k-\ell}(y)\widehat{n}_\ell(y)+\frac{1}{A}\sum_{\ell \neq 0} \pa_y \widehat{c}_{k-\ell}\pa_y \widehat{n}_\ell - \frac{1}{A}\sum_{\ell \neq 0}(k-\ell) \widehat{c}_{k-\ell} \ell \widehat{n}_\ell, \label{NL} \\
L_k:=&-\frac{1}{A}\overline{n}\widehat{n}_k-\frac{2}{A}({n}_0-\overline{n})\widehat{n}_{k}+\frac{1}{A}\pa_y c_0 \pa_y \widehat{n}_{k}+\frac{1}{A}\pa_y \widehat{c}_{k} \pa_y {n_0}. \label{L}
\end{align}
\end{subequations}
Here, the $L$ refers to ``linear with respect to the nonzero frequencies'' and $NL$ refers to ``nonlinear with respect to the nonzero frequencies''.

For constants $C_{ED},C_{L^2}, C_{\dot{H}^1},$ and $C_\infty$ determined by the proof, define $T_\star$ to be the end-point of the largest interval $[0,T_\star]$ such that the following hypotheses hold for all $T \leq T_\star$:

\begin{subequations}
(1) Nonzero mode $L_t^2 \dot{H}_{x,y}^1$ estimate:
\bel\ba\label{H1}
\frac{1}{A}\int_0^{T_\star}||\nabla n_{\neq}||_2^2dt\leq& 8||n_{in}||_2^2;\\
\ea\eel

(2) Nonzero mode enhanced dissipation estimate:
\bel\label{H2}\ba
||n_{\neq}(t)||_2^2\leq& 4C_{ED}||n_{in}||_{H^1}^2e^{-\frac{ct}{A^{1/2}\log A}}, \\
\ea\eel
where $c$ is a small constant depending only on $u$. 

(3) Uniform in time estimates on the zero mode:
\bel\ba\label{H3}
||n_0-\overline{n}||_{L^\infty_t(0,T_\star;L^2_{y})}\leq& 4C_{L^2},\\
||\pa_y n_0||_{L^\infty_t(0,T_\star;L^2_{y})}\leq& 4C_{\dot{H}^1};\\
\ea\eel

(4) $L^\infty$ estimate of the whole solution:
\bel\label{H4}\ba
||n||_{L^\infty_t(0,T_\star;L^\infty_{x,y})}\leq& 4C_\infty.
\ea\eel
\end{subequations}
Moreover, in order to simplify the exposition, we introduce the following constant:
\bel
\CC:=1+M+C_{ED}^{1/2}||n_{in}||_{H^1}+C_{L^2}+C_\infty.
\eel

\begin{remark}
In the above, $C_{ED}$ is first chosen depending only on $u$. Then, $C_{L^2}$ is chosen depending only on the initial data $n_{in}$ (and $C_{ED}$).
Then $C_{\infty}$ is chosen depending only on $n_{in}$, $C_{L^2}$, and $C_{ED}$. Finally, $C_{\dot{H}^1}$ depends on $n_{in}$, $C_{ED}$, $C_{L^2}$, and $C_\infty$.
Then, $A$ is chosen large depending on all of these parameters.
\end{remark}

We will refer to the hypotheses \eqref{H1}, \eqref{H2}, \eqref{H3}, and \eqref{H4} together as the \emph{bootstrap hypotheses}, denoted as \textbf{(H)}.
Notice that by local well-posedness of mild solutions, the quantities on the left-hand sides of \eqref{H1}, \eqref{H2}, \eqref{H3}, and \eqref{H4} take values continuously in time. Moreover, the inequalities are all satisfied with the $4$'s replaced by $2$'s for $t$ sufficiently small.
By the standard continuation criteria for \eqref{def:PKS}, the solution exists and remains smooth on an interval $(0,t_0]$, with $t_0 > T_\star$ such that $t_0 - T_\star$ can be taken to depend only on $\norm{n(T_\star)}_{L^2}$.
By continuity, the following proposition shows that the solution is global and satisfies the a priori estimates \textbf{(H)} for all time.

\begin{pro}\label{prop:boot2D}
For all $n_{in}$ and $u$, there exists an $A_0(u,\norm{n_{in}}_{H^1},\norm{n_{in}}_{L^\infty})$ such that if $A > A_0$ then the following conclusions, referred to as (\textbf{C}), hold on the interval $[0,T_\star]$:

\noindent
\begin{subequations}
(1) \bel\ba\label{ctrl:L2H1}
\frac{1}{A}\int_0^{T_\star}||\nabla_{x,y} {n}_{\neq}||_2^2dt\leq& 4||n_{in}||_2^2;\\
\ea\eel
(2) For all $t < T_\star$,
 \bel\label{ctrl:ED}\ba
||{n}_{\neq}(t)||_2^2\leq& 2C_{ED}||n_{in}||_{H^1}^2e^{-\frac{ct}{A^{1/2}\log A}}; \\
\ea\eel
(3) \bel\ba\label{ctrl:ZeroMd}
||{n}_0-\overline{n}||_{L^\infty_t(0,T_\star; L^2_{y})}\leq& 2C_{L^2},\\
||\pa_y {n}_0||_{L^\infty_t(0,T_\star;L^2_{y})}\leq& 2C_{\dot{H}^1};\\
\ea\eel
(4) \bel\label{ctrl:Linf}\ba
||n||_{L^\infty(0,T_\star;L^\infty_{x,y})}\leq& 2C_\infty.%(||n_{in}||_{H^1}, ||n_{in}||_1).
\ea\eel
\end{subequations}
\end{pro}
The remainder of the section is dedicated to proving Proposition \ref{prop:boot2D}.

We first point out that there is a uniform upper bound on $\norm{n(t)}_{L^2}$ over the initial time layer $t \leq \delta A$ for a sufficiently small $\delta$ depending only on $\norm{n_{in}}_{L^2}$ (as such we can always choose $A > \delta^{-1}$).
This is an immediate consequence of the standard local existence theory of \eqref{def:PKS} via the time-rescaling used in \eqref{def:PKSres}, however, we include a brief sketch of the a priori estimate for completeness.
Proposition \ref{C2pro0} and standard higher regularity theory for \eqref{def:PKS} (see e.g. \cite{JagerLuckhaus92}) imply that \eqref{ctrl:Linf} holds over the time interval $0 \leq t \leq \delta A$.

\begin{pro}\label{C2pro0}
For all $n_{in}\in L^2(\Torus^2)$, there exists $\delta = \delta(\norm{n_{in}}_{L^2})$ sufficiently small such that for $t\leq \delta A$, the following estimate holds,
\bel\label{initial time layer estimate}
\norm{n_{\neq}(t)}_{L^2}^2 + \norm{n_0 - \overline{n}}_{L^2}^2 =  ||n(t) - \overline{n}||_2^2\leq 2||n_{in}-\overline{n}||_{L^2}^2 \leq 2\norm{n_{in}}_{L^2}^2.
\eel
\end{pro}
\begin{proof}
The time derivative of the $L^2$ norm of $n-\overline{n}$ is estimated as follows, using a Gagliardo-Nirenberg-Sobolev inequality,
\be\ba
\frac{1}{2}\frac{d}{dt}||n-\overline{n}||_2^2=&-\frac{1}{A}||\na n||_2^2-\int \na\cdot(\na cn)(n-\overline{n})dx\\
=&-\frac{1}{A}||\na n||_2^2+\frac{1}{2A}||n-\overline{n}||_3^3+\frac{1}{A}\overline{n}||n-\overline{n}||_2^2\\
\leq &-\frac{1}{A}||\na n||_2^2+\frac{B}{A}||\na n||_2||n-\overline{n}||_2^2+\frac{1}{A}M||n-\overline{n}||_2^2\\
\lesssim& \frac{1}{A}||n-\overline{n}||_2^4+\frac{1}{A}M^2.
\ea\ee
The desired estimate follows (note that $M \lesssim \norm{n_{in}}_{L^2}$). 
\end{proof}

\subsection{Enhanced dissipation estimate, (\ref{ctrl:ED})} \label{sec:ED}
Proposition \ref{C2pro0} implies that \eqref{ctrl:ED} holds trivially on a time-scale like $t \lesssim A^{1/2}\log A$.
In order to deduce the enhanced dissipation effect for longer times, we use the hypocoercivity technique of \cite{BCZ15}, which builds on the earlier work of \cite{GallagherGallayNier2009,BeckWayne11}.
As outlined in \cite{villani2009}, hypocoercivity techinques are based on finding an energy which extracts the fact that the quadratic quantity $A^{-1}\norm{\grad f}^2_{L^2} + \norm{u' \partial_x f}_{L^2}^2$ is much `more coercive' than $A^{-1}\norm{\grad f}_{L^2}^2$.
In \cite{BCZ15} and here this is done via the following energies, defined $k$-by-$k$, 
\bel\label{Phi k}
\Phi_k[n(t)] =||\widehat{n}_k(t)||_2^2+||\sqrt{\al}\pa_y \widehat{n}_k(t)||_2^2+2kRe\lan i\beta u' \widehat{n}_k(t),\pa_y\widehat{n}_k
(t)\ran+|k|^2||\sqrt{\gamma}u' \widehat{n}_k(t)||_2^2;
\eel
\bel\label{Phi}
\Phi[n(t)]=\sum_{k\neq 0} \Phi_k[n(t)] =||n_{\neq}(t)||_{2}^2+||\sqrt{\al}\pa_y n_{\neq}(t)||_2^2 + 2\lan \beta u' \pa_xn_{\neq}(t),\pa_yn_{\neq}
(t)\ran+||\sqrt{\gamma}u'|\pa_x| n_{\neq}(t)||_2^2.
\eel
Here $\alpha,\beta$, and $\gamma$ are $k$-dependent constants (and hence should be interpreted as Fourier multipliers) satisfying
\begin{subequations} \label{def:abg}
\begin{align}
\alpha(A,k) & = \epsilon_\alpha A^{-1/2}\abs{k}^{-1/2} \\
\beta(A,k) & = \epsilon_\beta \abs{k}^{-1} \\
\gamma(A,k) & = \epsilon_\gamma A^{1/2}\abs{k}^{-3/2},
\end{align}
\end{subequations}
where $\epsilon_\alpha$, $\epsilon_\beta$, and $\epsilon_{\gamma}$ are small constants depending only on $u$ chosen in \cite{BCZ15}. Among other things, these are chosen such that 8$\beta^2 \leq \alpha \gamma$.  
Notice that in \cite{BCZ15} for treating general situations one must also take $\alpha,\beta$, and $\gamma$ to be $y$-dependent, however, as suggested by \cite{BeckWayne11}, this is not necessary to treat shear flows with non-degenerate critical points with $y \in \Torus$ or $y \in \Real$.
The parameters $\epsilon_\alpha$, $\epsilon_\beta$, and $\epsilon_{\gamma}$ are tuned such that,
\begin{align}
\Phi_k[n] \approx ||\widehat{n}_k||_2^2+||\sqrt{\al}\pa_y \widehat{n}_k||_2^2+|k|^2||\sqrt{\gamma}u'\widehat{n}_k||_2^2,
\end{align}
and hence
\begin{align}
\norm{\widehat{n}_k}_{2}^2 + A^{-1/2}\abs{k}^{-1/2}\norm{\partial_y \widehat{n}_k}_{2}^2 \lesssim \Phi_k[n] \lesssim \norm{\widehat{n}_k}_{2}^2 +  \abs{k}^{1/2} A^{1/2} \norm{\widehat{n}_k}_2^2 + A^{-1/2}\abs{k}^{-1/2}\norm{\partial_y \wh{n}_k}_2^2. \label{ineq:PhiEquiv}
\end{align}
As a result, $\Phi_k(t)$ is equivalent to the $H^1$ norm of $n_{k}$ but with constants that depend on $A$ and $k$.
The primary step in the results of \cite{BCZ15} is that for $u(y)$ satisfying the hypotheses in \eqref{thm:2D}, then for the passive scalar equation on $\mathbb{T}^2$,
\be
\pa_t f+u(y)\pa_x f= \frac{1}{A}\de f,
\ee
the norm $\Phi_k[f(t)]$ satisfies the following differential inequality for some small constant $\tilde{\ep}$  independent of $k, A$ (but depending on $u$):
\be
\frac{d}{dt}\Phi_k[f(t)]\leq -\tilde{\ep}\frac{|k|^{1/2}}{A^{1/2}}\Phi_k[f(t)].
\ee
The primary step in the proof of \eqref{ctrl:ED} is the analogous statement (though summed over all $k$ due to the nonlinearity).
\vs
\begin{pro}\label{C2pro}
There exists a small constant $c>0$ depending only on $u$ such that, under the bootstrap hypotheses and for $A$  sufficiently large, there holds
\bel\label{Phitestimate}
\frac{d}{dt}\Phi[n(t)]\leq -\frac{c}{A^{1/2}}\Phi[n(t)].
\eel
By \eqref{ineq:PhiEquiv}, it follows that
\begin{align}
\norm{n_{\neq}(t)}_{L^2}^2 \leq \Phi(0) e^{-cA^{-1/2}t} \lesssim A^{1/2}\norm{n_{in}}^2_{H^1}e^{-cA^{-1/2}t}.
\end{align}
\end{pro}
\begin{rmk}
Propositions \ref{C2pro0} and \ref{C2pro} together imply \eqref{ctrl:ED}.
Indeed, for $A$ sufficiently large:
\be\ba
||n_{\neq}(t)||_2^2 & \lesssim ||n_{in}||_{H^1}^2\mathbf{1}_{t\leq \frac{1}{2c}A^{1/2}\log A}+\mathbf{1}_{t\geq \frac{1}{2c}A^{1/2}\log A} A^{1/2}||n_{in}||_{H^1}^2 e^{-\frac{c}{A^{1/2}}t}\\
& \lesssim ||n_{in}||_{H^1}^2e^{-\frac{c}{2A^{1/2}\log A}t}.
\ea\ee
\end{rmk}

We first compute the time derivative of $\Phi_k[n(t)]$.

\begin{pro}\label{thm d/dt Phi k}
For $\tilde{\epsilon}$ sufficiently small depending only on $u$, there holds,
\begin{align}
\frac{d}{dt}\Phi_k[n(t)] \leq&\bigg\{-\frac{\tilde{\ep}}{2} \frac{|k|^{1/2}}{A^{1/2}}||\widehat{n}_k||_2^2-\frac{\tilde{\ep}}{2}\frac{|k|^{1/2}}{A^{1/2}}||\sqrt{\al}\pa_y \widehat{n}_k||_2^2-\frac{\tilde{\ep}}{2}\frac{|k|^{5/2}}{A^{1/2}}||\sqrt{\gamma}u'\widehat{n}_k||_2^2-\frac{1}{4A}||\pa_y\wh{n}_k||_2^2\nonumber\\
&-\frac{1}{2}|k|^2||\sqrt{\beta}u' \wh{n}_k||_2^2-\frac{1}{2A}|k|^2||\wh{n}_k||_2^2-\frac{1}{4A}||\sqrt{\al}\pa_{yy}\wh{n}_k||_2^2\nonumber\\
&-\frac{1}{4A}|k|^4||\sqrt{\gamma}u' \wh{n}_k||_2^2-\frac{1}{4A}|k|^2||\sqrt{\gamma}u'\pa_y\wh{n}_k||_2^2\bigg\}\nonumber\\
&+\bigg\{2Re\lan -L_k ,\widehat{n}_k\ran
-2Re\lan \al \pa_{yy}\widehat{n}_k,-L_k\ran-2kRe[\lan i\beta u'L_k,\pa_y\widehat{n}_k\ran+\lan i\beta u'\widehat{n}_k,\pa_y L_k\ran]\nonumber\\
&+2|k|^2Re\lan \gamma (u')^2\widehat{n}_k,-L_k\ran\bigg\}\nonumber\\
&+\bigg\{-2Re\lan NL_k,\widehat{n}_k\ran
+2Re\lan  \al \pa_{yy}\widehat{n}_k,NL_k\ran-2kRe[\lan i\beta u'NL_k,\pa_y\widehat{n}_k\ran+\lan i\beta u'\widehat{n}_k,\pa_y NL_k\ran]\nonumber\\
&-2|k|^2Re\lan \gamma (u')^2\widehat{n}_k,NL_k\ran\bigg\}\nonumber\\
=:&\mathcal{N}_k+\{L_k^1+L_k^{\al}+L_k^{\beta}+L_k^{\gamma}\}+\{NL_k^1+NL_k^{\al}+NL_k^{\beta}+NL_k^{\gamma}\}, \label{ddtPhik}
\end{align}
where $\mathcal{N}_k$ refers to the negative terms.
Recall that $L_k,NL_k$ are defined in (\ref{L},\ref{NL}).
\end{pro}
\begin{proof}
The estimates from the linear terms (that is, the terms arising from the passive scalar equation \eqref{def:passcale}) are made in \cite{BCZ15} and are omitted for the sake of brevity.
The extra terms from the Keller-Segel nonlinearity in \eqref{def:PKSres} are immediate.
\end{proof}

The remainder of the section is devoted to controlling $L$ and $NL$ by the negative terms in (\ref{ddtPhik}).

\subsubsection{Estimate on the $L$ terms in \eqref{ddtPhik}}
These terms are linear in the $k$-th mode, and it accordingly makes sense to estimate these terms $k$-by-$k$.
In this section we prove that for $A$ sufficiently large,
\bel\label{Lestimate}
L_k^1+L_k^\al+L_k^\beta+L_k^\gamma \leq -\frac{1}{4}\mathcal{N}_k.
\eel
We begin by estimating the $L_k^1$ term in \eqref{ddtPhik}.
Integrating by parts and using Lemma \ref{lem 2}, Lemma \ref{lem 3}, and the bootstrap hypotheses, we have, for any fixed constant $B \geq 1$, 
\be\ba
|L_k^1| \leq& \frac{2}{A}(2||{n}_0-\overline{n}||_\infty+\overline{n})||\widehat{n}_k||_2^2+\frac{1}{AB}||\pa_y \widehat{n}_k||_2^2+\frac{B}{A}||\pa_y{c}_0||_\infty^2||\widehat{n}_k||_2^2
\\ & +\bigg|Re\frac{2}{A}\lan \pa_{yy}\widehat{c}_k\widehat{n}_k + \pa_y \widehat{c}_k \partial_y \widehat{n}_k, n_0 - \overline{n} \ran \bigg| \\  
\lesssim & \frac{C_{2,\infty}^2}{A}||\widehat{n}_k||_2^2+\frac{1}{AB}||\pa_y \widehat{n}_k||_2^2.
\ea\ee
Therefore, we can choose $B$ sufficiently large, and then $A$ sufficiently large, such that the following holds:
\begin{align*}
\abs{L_k^1} \lesssim \frac{C_{2,\infty}^2}{A}\norm{\widehat{n}_k}_{2}^2 + \frac{1}{A B}\norm{\partial_y\widehat{n}_k}_2^2 \leq -\frac{1}{16}\mathcal{N}_k,
\end{align*}
and hence by the definition of $\mathcal{N}_k$, this is consistent with \eqref{Lestimate}.

Turn next to $L_k^\al$ in \eqref{ddtPhik}, which we divide into the following four contributions:
\bel\label{L al}\ba
L_k^\al & =-2Re\lan \al \pa_{yy}\widehat{n}_k,\frac{1}{A}\overline{n}\widehat{n}_k+\frac{2}{A}({n}_0-\overline{n})\widehat{n}_k-\frac{1}{A}\pa_y{c}_0\pa_y\widehat{n}_k-\frac{1}{A}\pa_y\widehat{c}_k\pa_y{n}_0\ran \\ & =: L^\al_{k,0}+L^\al_{k,1}+L^\al_{k,2}+L^\al_{k,3}.
\ea\eel
For the $L^\al_{k,1}$ term, we have the following by the bootstrap hypotheses, for any fixed $B \geq 1$:
\be\ba
\abs{L^\al_{k,1}} & \lesssim \frac{1}{AB}||\sqrt{\al}\pa_{yy}\widehat{n}_k||_2^2+\frac{B}{A}||\sqrt{\al}({n}_0-\overline{n})\widehat{n}_k||_2^2\\
& \lesssim \frac{1}{AB}||\sqrt{\al}\pa_{yy}\widehat{n}_k||_2^2 + \frac{B}{A^{3/2}}||{n}_0-\overline{n}||_\infty^2||\widehat{n}_k||_2^2 \\
& \lesssim \frac{1}{AB}||\sqrt{\al}\pa_{yy}\widehat{n}_k||_2^2 + \frac{B C_\infty^2}{A^{3/2}}||\widehat{n}_k||_2^2.
\ea\ee
Recalling the definition $\mathcal{N}_k$ from \eqref{ddtPhik}, it follows that by choosing $B$, then $A$, sufficiently large, we can control this term consistent with \eqref{Lestimate}.
The $L^\al_{k,0}$ term is treated in the same manner; we omit the details for brevity.

Next, we estimate the second term $L^\al_{k,2}$ in (\ref{L al}).
Using Lemma \ref{lem 3}, we have the following for any $B \geq 1$:
\be\ba
\abs{L^\al_{k,2}} \lesssim &\frac{1}{BA}||\sqrt{\al}\pa_{yy}\widehat{n}_k||_2^2+\frac{B}{A^{3/2}}||\pa_y{c}_0||_\infty^2||\pa_y\widehat{n}_k||_2^2\\
\lesssim &\frac{1}{BA}||\sqrt{\al}\pa_{yy}\widehat{n}_k||_2^2+\frac{B\CC^2}{A^{3/2}}||\pa_y\widehat{n}_k||_2^2.
\ea\ee
Hence, by the bootstrap hypotheses and the definition of $\mathcal{N}_k$, it follows we can choose $B$ large and then $A$ large to control this term consistent with \eqref{Lestimate}.

Similarly, for $L^\al_{k,3}$ in (\ref{L al}), by Lemma \ref{lem 2}
\be\ba
|L^\al_{k,3}| \lesssim&\frac{1}{BA}||\sqrt{\al}\pa_{yy}\widehat{n}_k||_2^2+\frac{B}{A^{3/2}}||\pa_y\widehat{c}_k||_\infty^2||\pa_y{n}_0||_2^2\\
\lesssim& \frac{1}{BA}||\sqrt{\al}\pa_{yy}\widehat{n}_k||_2^2+\frac{B}{A^{3/2}}||\widehat{n}_k||_2^2||\pa_y{n}_0||_2^2.\\
\ea\ee
As above, it follows we can choose $B$ large and then $A$ large to control this term consistent with \eqref{Lestimate}.

Next, turn to the $L_k^\beta$ term in (\ref{ddtPhik}), which we divide into two contributions:
\bel\label{L beta}\ba
L_k^\beta=&2kRe \lan i\beta u'\widehat{n}_k,\pa_y(\frac{1}{A}\overline{n}\widehat{n}_k+\frac{2}{A}({n}_0-\overline{n})\widehat{n}_k-\frac{1}{A}\pa_y{c}_0\pa_y\widehat{n}_k-\frac{1}{A}\pa_y\widehat{c}_k\pa_y{n}_0)\ran  \\
&+2k Re \lan i\beta u'(\frac{1}{A}\overline{n}\widehat{n}_k+\frac{2}{A}({n}_0-\overline{n})\widehat{n}_k-\frac{1}{A}\pa_y{c}_0\pa_y\widehat{n}_k-\frac{1}{A}\pa_y\widehat{c}_k\pa_y{n}_0),\pa_y\widehat{n}_k\ran \\
=:& L^\beta_{k,1}+L^\beta_{k,2}.
\ea\eel
By analogy with the $\alpha$ terms, the first term in (\ref{L beta}) is further decomposed via
\bel\label{L beta 1}\ba
L^\beta_{k,1}=&2kRe\lan i\beta u'\widehat{n}_k,\pa_y(\frac{1}{A}\overline{n}\widehat{n}_k+\frac{2}{A}({n}_0-\overline{n})\widehat{n}_k-\frac{1}{A}\pa_y{c}_0\pa_y\widehat{n}_k-\frac{1}{A}\pa_y\widehat{c}_k\pa_y{n}_0)\ran\\
=:&L^\beta_{k,10}+L^\beta_{k,11}+L^\beta_{k,12}+L^\beta_{k,13}.
\ea\eel
For the $L^\beta_{k,11}$ term in (\ref{L beta 1}), we have the following, (for any fixed $B \geq 1$),
\be\ba
\abs{L^\beta_{k,11}}=&\abs{2kRe\lan i\beta u''\widehat{n}_k+i\beta u'\pa_y\widehat{n}_k,\frac{2}{A}({n}_0-\overline{n})\widehat{n}_k\ran}\\
=&\abs{2kRe\lan i\beta u'\pa_y\widehat{n}_k,\frac{2}{A}({n}_0-\overline{n})\widehat{n}_k\ran}\\
\lesssim &\frac{1}{AB}||\pa_y\widehat{n}_k||_2^2+\frac{B}{A}|k|^2||{n}_0-\overline{n}||_\infty^2||\sqrt{\beta}u'\widehat{n}_k||_2^2.
\ea\ee
By the bootstrap hypotheses and by choosing $B$, then $A$, large enough, this term is controlled consistent with \eqref{Lestimate}.
The $L^\beta_{k,10}$ term is treated in the same manner; we omit the details for the sake of brevity.

For the $L^\beta_{k,12}$ term in (\ref{L beta 1}), using Lemma \ref{lem 3}, we have that for some fixed $B \geq 1$, the following holds,
\be
\ba
|L^\beta_{k,12}| \leq &\bigg|2kRe\lan i\beta u' \widehat{n}_k,\frac{1}{A}({n}_0-\overline{n})\pa_y\widehat{n}_k\ran\bigg| +\bigg|2kRe\lan i\beta u'\widehat{n}_k,\frac{1}{A}\pa_y{c}_0\pa_{yy}\widehat{n}_k\ran\bigg|\\
\leq&\frac{1}{AB}||\pa_y\widehat{n}_k||_2^2+\frac{B|k|^2 \beta}{A}||{n}_0-\overline{n}||_\infty^2||\sqrt{\beta}u' \widehat{n}_k||_2^2\\
&+\frac{1}{AB}||\sqrt{\al}\pa_{yy}\widehat{n}_k||_2^2+\frac{B|k|^2\beta}{A\alpha}||\sqrt{\beta}u' \widehat{n}_k||_2^2 ||\pa_y{c}_0||_\infty^2\\
\lesssim &\frac{1}{AB}||\pa_y\widehat{n}_k||_2^2+\frac{B|k|^2}{A}||{n}_0-\overline{n}||_\infty^2||\sqrt{\beta}u' \widehat{n}_k||_2^2\\
&+\frac{1}{AB}||\sqrt{\al}\pa_{yy}\widehat{n}_k||_2^2+\frac{B|k|^2 M^2}{A^{1/2}}||\sqrt{\beta}u' \widehat{n}_k||_2^2 
\ea
\ee
As above, by the bootstrap hypotheses, for $B$ and $A$ sufficiently large, this term is controlled consistent with \eqref{Lestimate}.

Consider next $L^\beta_{k,13}$ in (\ref{L beta 1}), which we integrate by parts and further sub-divide as:
\bel\label{L beta 13}\ba
L^\beta_{k,13}=&2kRe\lan i\beta u''  \widehat{n}_k+i\beta u'\pa_y  \widehat{n}_k,\frac{1}{A}\pa_{y}\widehat{c}_k\pa_y{n}_0\ran =: L^\beta_{k,131}+L^\beta_{k,132}.
\ea\eel
For $L^\beta_{k,131}$, by Lemma \ref{lem 2} and the definition of $\beta$, we have the following for a large constant $B \geq 1$,  
\be\ba
\abs{L^\beta_{k,131}}\leq&\frac{|k|^2B}{A}||\beta u'' \widehat{n}_k||_2^2||\pa_y{n}_0||_2+\frac{1}{AB}||\pa_y \widehat{c}_k||_\infty^2||\pa_y {n}_0||_2\\
\lesssim & \frac{B||\pa_y {n}_0||_2}{A}||\widehat{n}_k||_2^2+\frac{1}{AB}||\widehat{n}_k||_2^2||\pa_y {n}_0||_2.
\ea\ee
Therefore, by the bootstrap hypotheses, for $B$, then $A$, large, this term is controlled consistent with \eqref{Lestimate}.
Using Lemma \ref{lem 2} and the definition of $\beta$, the $L^\beta_{k,132}$ term in (\ref{L beta 13}) is handled as follows for a large constant $B \geq 1$:
\bel\label{iii 3 I}\ba
\abs{L^\beta_{k,132}} \lesssim & \frac{1}{A B}||\pa_y \widehat{n}_k||_2^2+\frac{B}{A}||\widehat{n}_k||_2^2||\pa_y{n}_0||_2^2. 
\ea\eel
Therefore, by the bootstrap hypotheses (in particular, \eqref{H3}), for $B$ and $A$ sufficiently large, this is consistent with \eqref{Lestimate}.

Turn next to $L^\beta_{k,2}$, which we sub-divide as follows:
\bel\label{L beta 2}\ba
L^\beta_{k,2} =&2kRe\lan i\beta u'\frac{1}{A}\overline{n}\widehat{n}_k,\pa_y\widehat{n}_k\ran+2kRe\lan i\beta u'\frac{2}{A}({n}_0-\overline{n})\widehat{n}_k,\pa_y\widehat{n}_k\ran\\
&-2kRe\lan i\beta u'\frac{1}{A}\pa_y{c}_0\pa_y\widehat{n}_k,\pa_y\widehat{n}_k\ran-2kRe\lan i\beta u'\frac{1}{A}\pa_y\widehat{c}_k\pa_y{n}_0,\pa_y\widehat{n}_k\ran\\
=:&L^\beta_{k,20}+L^\beta_{k,21}+L^\beta_{k,22}+L^\beta_{k,23}.
\ea\eel
By anti-symmetry,  $L^\beta_{k,22}=0$ (note this is simply the observation that $\brak{u'\partial_y c_0 \sqrt{\beta} \partial_x \partial_yn, \sqrt{\beta}\partial_y n} = 0$).
For the $L^\beta_{k,21}$ term, we use the following straightforward estimate for a constant $B \geq 1$:
\be\ba
\abs{L^\beta_{k,21}} \lesssim &\frac{1}{AB}||\pa_y\widehat{n}_k||_2^2+\frac{B||{n}_0-\overline{n}||_\infty^2}{A}|k|^2||\sqrt{\beta}u'\widehat{n}_k||_2^2.
\ea\ee
This is consistent with \eqref{Lestimate} by the bootstrap hypotheses and $B$,$A$ large. The $L^\beta_{k,20}$ is treated similarly, we skip the detail for brevity.
The $L^\beta_{k,23}$ term can be estimated in the same manner as $L^\beta_{k,132}$ above in (\ref{iii 3 I}) and hence is omitted for brevity.
This completes the treatment of the $L_k^\beta$ term in \eqref{ddtPhik}.

Finally, we estimate $L_k^\gamma$ in (\ref{ddtPhik}).
We first sub-divide into four parts:
\bel\label{L gamma}\ba
L_k^\gamma & = 2|k|^2Re\lan \gamma (u')^2\widehat{n}_k,\frac{1}{A}\overline{n}\widehat{n}_k+\frac{2}{A}({n}_0-\overline{n})\widehat{n}_k-\frac{1}{A}\pa_y{c}_0\pa_y\widehat{n}_k-\frac{1}{A}\pa_y\widehat{c}_k\pa_y{n}_0\ran \\ & =:  L^\gamma_{k,0}+ L^\gamma_{k,1}+L^\gamma_{k,2}+L^\gamma_{k,3}.
\ea\eel
The second term in (\ref{L gamma}) is estimated as follows for a fixed constant $B \geq 1$,
\be\ba
\abs{L^\gamma_{k,1}} \lesssim &\frac{B\gamma}{A\beta}||\sqrt{\beta}u' \widehat{n}_k||_2^2||{n}_0-\overline{n}||_\infty^2 +\frac{|k|^4}{AB}||\sqrt{\gamma}u'\widehat{n}_k||_2^2\\
\lesssim &\frac{B}{A^{1/2}}||\sqrt{\beta}u' \widehat{n}_k||_2^2||{n}_0-\overline{n}||_\infty^2+\frac{|k|^4}{AB}||\sqrt{\gamma}u'\widehat{n}_k||_2^2.
\ea\ee
As above, this is consistent with \eqref{Lestimate} by the bootstrap hypotheses and $B$,$A$ large. The term $L^\gamma_{k,0}$ is treated similarly, hence, we omit the details for the sake of brevity.
The term $L^\gamma_{k,2}$ in (\ref{L gamma}) is similar. Indeed, by Lemma \ref{lem 3}, we have for $B \geq 1$ large,
\be\ba
\abs{L^\gamma_{k,2}} \lesssim &\frac{B\gamma }{A\beta }|k|^2||\sqrt{\beta}u'\widehat{n}_k||_2^2||\pa_y{c}_0||_\infty^2+\frac{|k|^2}{AB}||\sqrt{\gamma}u'\pa_y\widehat{n}_k||_2^2\\
\lesssim &\frac{B}{A^{1/2}}|k|^2||\sqrt{\beta}u'\widehat{n}_k||_2^2||{n}_0-\overline{n}||_2^2+\frac{|k|^2}{AB}||\sqrt{\gamma}u'\pa_y\widehat{n}_k||_2^2.
\ea\ee
As usual, this is consistent with \eqref{Lestimate} by the bootstrap hypotheses and $B$,$A$ large.
The $L^\gamma_{k,3}$ term in (\ref{L gamma}), is estimated slightly differently; using Lemma \ref{lem 2}, we have for $B \geq 1$ large,
\be\ba
\abs{L^\gamma_{k,3}} \lesssim &\frac{1}{A^{1/2}B}|k|^{5/2}||\sqrt{\gamma}u'\widehat{n}_k||_2^2+\frac{B}{A^{3/2}}|k|^{3/2}\gamma ||\pa_y\widehat{c}_k||_\infty^2||\pa_y{n}_0||_2^2\\
\lesssim &\frac{1}{A^{1/2}B}|k|^{5/2}||\sqrt{\gamma}u'\widehat{n}_k||_2^2+\frac{B}{A}||\widehat{n}_k||_2^2||\pa_y{n}_0||_2^2.
\ea\ee
This is consistent with \eqref{Lestimate} by the bootstrap hypotheses and $B$,$A$ large.
This completes the proof of \eqref{Lestimate}, and hence, under the bootstrap hypotheses, the contributions of the $L$ terms in \eqref{ddtPhik} is absorbed by the $\mathcal{N}_k$ terms for $A$ chosen sufficiently large.

\subsubsection{Estimate on $NL$ terms} \label{sec:NLED}
As these terms are nonlinear in non-zero frequencies, it is more natural to consider all of the frequencies at once. For the $NL_k^1$ term in \eqref{ddtPhik}, writing,
\begin{align*}
-\sum_{k \neq 0}2 Re\lan NL_k,\widehat{n}_k\ran = -\brak{\frac{1}{A}\grad \cdot \left( n_{\neq} \grad c_{\neq}\right),n_{\neq}} = \frac{1}{A} \brak{ n_{\neq} \grad c_{\neq},\grad n_{\neq}} 
\leq \frac{1}{A} \norm{\grad c_{\neq}}_{\infty }\norm{\grad n_{\neq}}_{2} \norm{n_{\neq}}_{2}. 
\end{align*}
By \eqref{elliptic estimate appendix}, for some constant $B > 0$,
\begin{align*}
-\sum_{k \neq 0}2 Re\lan NL_k,\widehat{n}_k\ran 
& \lesssim \frac{1}{AB}\norm{\grad n_{\neq}}_{2}^2 + \frac{B}{A}\CC^{2}\norm{n_{\neq}}_{2}^2 .
\end{align*}
By first choosing $B$ large relative to the implicit constant, and then choosing $A$ large (relative to constants and $B$), these terms are absorbed by the negative terms in \eqref{ddtPhik}.

For the $NL_k^{\al}$ term in \eqref{ddtPhik}, we use \eqref{elliptic estimate appendix} and the bootstrap hypotheses to deduce (using the definition of $\alpha$;  recall that $\alpha$ is a Fourier multiplier in $x$),
\begin{align*}
2Re \sum_{k \neq 0} \brak{\alpha(\partial_x) \partial_{yy} \widehat{n}_k, NL_k} & = \frac{2}{A} \brak{\alpha(\partial_x) \pa_{yy} n_{\neq}, \grad \cdot \left( n_{\neq} \grad c_{\neq}\right)} \\
& \lesssim \frac{1}{A^{5/4}} \norm{\sqrt{\al}\pa_{yy} n_{\neq}}_2\left(\norm{\grad n_{\neq}}_2 \norm{\grad c_{\neq}}_{\infty} + \norm{n_{\neq}}_2 \norm{n_{\neq}}_{\infty}\right) \\
& \lesssim \frac{1}{A^{5/4}} \norm{\sqrt{\al}\pa_{yy} n_{\neq}}_2\left(\CC \norm{\grad n_{\neq}}_2 + C_\infty\norm{n_{\neq}}_2\right) \\
& \lesssim \frac{1}{A^{5/4}} \norm{\sqrt{\al}\pa_{yy} n_{\neq}}_2^2+\frac{\CC^2}{A^{5/4}}\norm{\grad n_{\neq}}_2^2  + \frac{C_\infty^2}{A^{5/4}}\norm{n_{\neq}}_2^2,
\end{align*}
and choosing $A$ large, these terms are absorbed by the negative terms in \eqref{ddtPhik}.

There are two terms in $NL_k^{\beta}$ in \eqref{ddtPhik}; we estimate the first as follows (using that $\beta(k) \lesssim \abs{k}^{-1}$ and defines a self-adjoint operator, Lemma \ref{elliptic estimate appendix}, and that $u$ does not depend on $x$):
\begin{align}
-2k\sum_{k \neq 0} Re\brak{i\beta(\partial_x) u'NL_k,\pa_y \widehat{n}_k} & =- \frac{2}{A}\brak{\beta(\partial_x) u' \pa_x\grad\cdot(n_{\neq} \grad c_{\neq}),\pa_y n_{\neq}} \nonumber \\
& \lesssim \frac{1}{A}\norm{n_{\neq}}_2 \norm{n_{\neq}}_{\infty}\norm{\partial_y n_{\neq}}_2 + \frac{1}{A}\norm{\beta u' \pa_x\pa_y n_{\neq}}_2 \norm{\grad c_{\neq}}_{\infty}\norm{\na n_{\neq}}_2 \nonumber\\
& \lesssim  \frac{C_\infty}{A}\norm{n_{\neq}}_2 \norm{\pa_y n_{\neq}}_2 + \frac{\CC}{A^{5/4}}\norm{\sqrt{\gamma} u' \pa_x\pa_y n_{\neq}}_2\norm{\na n_{\neq}}_2 \nonumber\\
& \lesssim \frac{C_\infty^2}{A^{3/4}}\norm{n_{\neq}}_2^2 + \frac{1}{A^{5/4}}\norm{\grad n_{\neq}}_2^2 + \frac{C_{2,\infty}^2}{A^{5/4}}\norm{\sqrt{\gamma} u' \pa_x\pa_y n_{\neq}}_2^2. \label{firstbetaNL}
\end{align}
Choosing $A$ large, these terms are absorbed by the negative terms in \eqref{ddtPhik}.
For the second term in $NL_k^{\beta}$ we use
\begin{align*}
2Re\sum_{k \neq 0} \brak{\beta(\pa_x) u' ik\wh{n}_{k},\partial_y NL_k} & = \frac{2}{A}\brak{\beta(\partial_x) u'\pa_x n_{\neq}, \pa_y \grad\cdot(n_{\neq} \grad c_{\neq})}  \\
& = - \frac{2}{A}\brak{\beta(\partial_x) u'' \pa_x n_{\neq}, \grad\cdot(n_{\neq} \grad c_{\neq})} - \frac{2}{A}\brak{\beta(\partial_x) u' \pa_x\pa_y n_{\neq}, \grad\cdot(n_{\neq} \grad c_{\neq})}  \\
& = NL^\beta_{k,1} + NL^\beta_{k,2}.
\end{align*}
Using Lemma \ref{elliptic estimate appendix}, $\beta(\partial_x) = \epsilon_\beta \abs{\partial_x}^{-1}$, and  that $u$ does not depend on $x$, we have,
\begin{align*}
\abs{NL^\beta_{k,1}} & \lesssim \frac{1}{A}\norm{\beta u''\pa_xn_{\neq}}_2 \left(\norm{\grad n_{\neq}}_2\norm{\grad c_{\neq}}_{\infty} + \norm{n_{\neq}}_2\norm{n_{\neq}}_\infty\right) \\
& \lesssim \frac{\CC}{A}\norm{n_{\neq}}_2 \left(\norm{\grad n_{\neq}}_2 + \norm{n_{\neq}}_2\right) \\
& \lesssim \frac{1}{BA}\norm{\grad n_{\neq}}_2^2 + \frac{B\CC^2}{A} \norm{n_{\neq}}_2^2,
\end{align*}
yielding terms which are absorbed by the negative terms in \eqref{ddtPhik} for $A$ sufficiently large. The treatment of $ NL^\beta_{k,2}$ is similar to \eqref{firstbetaNL}, hence it is omitted for the sake of brevity.

Turn finally to term $NL_k^{\gamma}$ in \eqref{ddtPhik} associated with $\gamma$:
\begin{align}
-2Re\sum_{k \neq 0} \brak{\abs{k}^2\gamma(k) u' n_{k}, u' NL_k} & = -\frac{2}{A}\brak{\gamma(\partial_x) u' \partial_x n_{\neq}, u' \partial_x \grad\cdot(n_{\neq} \grad c_{\neq})} \nonumber\\
& = \frac{2}{A}\brak{\gamma(\partial_x) u' \partial_x \grad n_{\neq}, u' \partial_x(n_{\neq} \grad c_{\neq})} +\frac{4}{A}\brak{\gamma(\partial_x) u' u'' \partial_{x} n_{\neq}, \partial_x(n_{\neq} \partial_y c_{\neq})}\nonumber \\
&=:NL^\gamma_{k,1}+NL^\gamma_{k,2}.\label{NL gamma}
\end{align}
 Then we use $\gamma(\partial_x) = \epsilon_\gamma A^{1/2} \abs{\partial_x}^{-3/2}$, interpolation, and Lemma \ref{elliptic estimate appendix} to deduce the following bound for $NL^\gamma_{k,1}$:
\begin{align*}
NL^\gamma_{k,1}\lesssim&\frac{1}{A}||\sqrt{\gamma}u'\pa_x\na n_{\neq}||_2||\sqrt{\gamma}\pa_x(u'n_{\neq}\na c_{\neq})||_2\\
 \lesssim &\frac{1}{A^{3/4}}\norm{\sqrt{\gamma} u' \partial_x \grad n_{\neq}}_2\norm{\abs{\partial_x}^{1/4}(u' n_{\neq} \grad c_{\neq})}_2 \\
 \lesssim& \frac{1}{A^{3/4}}\norm{\sqrt{\gamma} u' \partial_x \grad n_{\neq}}_2 \norm{ u' n_{\neq}\grad c_{\neq}}_2^{3/4}\norm{\partial_x(u' n_{\neq}\grad c_{\neq})}_2^{1/4} \\
 \lesssim& \frac{1}{A^{3/4}}\norm{\sqrt{\gamma} u' \partial_x \grad n_{\neq}}_2 \norm{ u' n_{\neq}}_2^{3/4} \norm{\grad c_{\neq}}_{\infty}^{3/4}\left(\norm{u'\partial_x n_{\neq}}_2^{1/4} \norm{\grad c_{\neq}}_{\infty}^{1/4} + \norm{n_{\neq}}_{\infty}^{1/4}\norm{\partial_x \grad c_{\neq}}_2^{1/4}\right) \\
 \lesssim& \frac{\CC}{A^{3/4}}\norm{\sqrt{\gamma} u' \partial_x \grad n_{\neq}}_2\norm{ u' n_{\neq}}_2^{3/4}\left(\norm{u'\partial_x n_{\neq}}_2^{1/4} + \norm{n_{\neq}}_2^{1/4}\right) \\
\lesssim& \frac{1}{BA}\norm{\sqrt{\gamma} u' \partial_x \grad n_{\neq}}_2^2 + \frac{B \CC^{2}}{A^{1/2}}\norm{u' n_{\neq}}_2^{3/2}\left(\norm{u'\partial_x n_{\neq}}_2^{1/2} + \norm{n_{\neq}}_2^{1/2}\right).
\end{align*}
Hence, for $B$ chosen large, then $A$ chosen large, we may absorb these contributions in the negative terms in \eqref{ddtPhik}.

Next we estimate the $NL^\gamma_{k,2}$ term in (\ref{NL gamma}),
\begin{align*}
NL^\gamma_{k,2}\lesssim&\frac{1}{A^{3/4}}||\sqrt{\gamma} u'|\pa_x|^{5/4}n_{\neq}||_2||n_{\neq}\na c_{\neq}||_2 \lesssim \frac{1}{A^{3/4}}||\sqrt{\gamma} u'|\pa_x|^{5/4}n_{\neq}||_2^2+\frac{\CC^2}{A^{3/4}}||n_{\neq}||_2^2.
\end{align*}
Hence, for $A$ chosen large, we may absorb these contributions in the negative terms in \eqref{ddtPhik}. This finishes the estimate of the $NL$ terms.

\subsection{Nonzero mode $L_t^2 \dot{H}_{x,y}^1$ estimate \eqref{ctrl:L2H1}}

The nonzero mode $L_t^2 \dot{H}_{x,y}^1$ estimate \eqref{ctrl:L2H1} comes from an estimate on the $\frac{d}{dt}||\wh{n}_{\neq}||_2^2$ and the knowledge that $||\wh{n}_{\neq}||_2^2$ is bounded by $4C_{ED}||n_{in}||_{H^1}^2$ from Hypothesis \eqref{H2}.
Indeed, from \eqref{def:PKS} and Lemma \ref{lem 3}, there holds for some universal constant $B$,
\begin{align}
\frac{1}{2}\frac{d}{dt}||{n}_{\neq}||_2^2=&\brak{n_{\neq},\frac{1}{A}\de n_{\neq}+\frac{1}{A}\overline{n}n_{\neq} + \frac{1}{A}\left(n_0-\bar{n}\right)n_{\neq}-\frac{1}{A}\na c_0\cdot \na n_{\neq}-\frac{1}{A}\na c_{\neq}\cdot \na n_0-\frac{1}{A}(\na\cdot(\na c_{\neq} n_{\neq}))_{\neq}}\nonumber\\
\leq & -\frac{1}{2A}||\na n_{\neq}||_2^2+\frac{1}{A}||{n}_{\neq}||_2^2\left(||n_0-\overline{n}||_\infty + \norm{n_0}_{\infty} + \norm{\grad c_0}_2^2\right) +\frac{1}{A}||\partial_y c_{\neq}||_\infty ||\pa_yn_0||_2 \norm{n_{\neq}}_2 \nonumber \\
& + \frac{2}{A} \norm{\grad n_{\neq}}_2 ||\na c_{\neq}||_4 ||n_{\neq}||_4 \nonumber \\ 
 \leq &-\frac{1}{2A}||\na n_{\neq}||_2^2+\frac{B(\CC^2+C_{\dot{H}^1}^2)}{A}||n_{\neq}||_2^2 + \frac{2}{A} \norm{\grad n_{\neq}}_2 ||\na c_{\neq}||_4 ||n_{\neq}||_4. 
\label{d dt n neq 2 2}
\end{align}
By the Gagliardo-Nirenberg-Sobolev inequality, we obtain
\begin{align*}
\norm{\grad n_{\neq}}_2 ||\na c_{\neq}||_4 ||n_{\neq}||_4  & \lesssim ||c_{\neq}||_2^{1/4}||\na^2 c_{\neq}||_2^{3/4}||n_{\neq}||_2^{1/2}||\na n_{\neq}||_2^{3/2} \lesssim ||n_{\neq}||_2^{3/2}||\na n_{\neq}||_2^{3/2},
\end{align*}
which implies the following (possibly adjusting $B$),
\begin{align}
\frac{1}{2}\frac{d}{dt}||{n}_{\neq}||_2^2 & \leq -\frac{1}{4A}||\na n_{\neq}||_2^2+\frac{B(\CC^2+C_{\dot{H}^1}^2)}{A}||n_{\neq}||_2^2 + \frac{B}{A}\norm{n_{\neq}}_2^6. \label{ineq:dtneq}
\end{align}
By \eqref{H2}, the time integral of $\frac{1}{A}||n_{\neq}||_{2}^2$ is estimated as
\bel
\int_0^{T_\star}\frac{1}{A}||n_{\neq}(t)||_2^2 dt \lesssim \frac{\log A}{A^{1/2}}.
\eel
Hence, by applying \eqref{H1},  integrating \eqref{ineq:dtneq}, and choosing $A$ large, there holds
\be
\ba
\frac{1}{A}\int_0^{T_\star}||\na n_{\neq}||_2^2dt\leq \frac{1}{A^{1/4}}+2||n_{in}||_{2}^2\leq 4||n_{in}||_{2}^2.
\ea
\ee
As a result, we have proved \eqref{ctrl:L2H1}.

\subsection{Zero mode estimate \eqref{ctrl:ZeroMd}}\label{G(t)trick}
First, by non-negativity, note that $\norm{n_0}_{L^1} = \norm{n}_{L^1} = M$ is constant in time.
We begin by estimating $||{n}_0-\overline{n}||_2^2$, then go on to estimate $||\pa_y {n}_0||_2^2$.
From \eqref{def:PKS} we have, by Minkowski's inequality, 
\be\ba
\frac{1}{2}\frac{d}{dt}||n_0-\overline{n}||_2^2 =&\lan n_0-\overline{n},\frac{1}{A}\pa_{yy}{n}_0-\frac{1}{A}\pa_y(\pa_y{c}_0{n}_0)-\frac{1}{A}(\na\cdot(\na c_{\neq} n_{\neq}))_0\ran\\
=&-\frac{1}{A}||\pa_y {n}_0||_2^2+\frac{1}{A}\lan \pa_y{n}_0,\pa_y {c}_0 n_0\ran+\frac{1}{A}\lan \pa_y {n}_0,(\pa_y c_{\neq} n_{\neq})_0\ran\\
\leq&-\frac{1}{2A}||\pa_yn_0||_2^2+\frac{1}{A}||\pa_y {c}_0||_\infty^2||{n}_0||_2^2+\frac{1}{A}||(\pa_y {c}_{\neq}{n}_{\neq})_0||_{L^2(\mathbb{T})}^2\\
\leq&-\frac{1}{2A}||\pa_y{n}_0||_2^2+\frac{B M^2}{A} ||{n}_0||_2^2+\frac{1}{A}||\pa_y {c}_{\neq}{n}_{\neq}||_{L^2(\mathbb{T}^2)}^2.\\
\ea\ee
Recall the following Nash inequality on $\mathbb{T}$, under the assumption that $\int_{\mathbb{T}} \rho dx=0$:
\bel\label{Nash ineq}
||\rho||_{L^2(\mathbb{T})}\lesssim ||\rho||_{L^1(\mathbb{T})}^{2/3}||\pa_y\rho||_{L^2(\mathbb{T})}^{1/3}.
\eel
Hence,
\be
-||\pa_y{n}_0||_2^2\lesssim-\frac{||{n}_0-\overline{n}||_2^6}{||n-\overline{n}||_1^4}\lesssim -\frac{||{n}_0-\overline{n}||_2^6}{M^4}.
\ee
Therefore,
\be\ba
\frac{d}{dt}|| {n}_0-\overline{n}||_2^2\lesssim&-\frac{1}{A}\frac{||{n}_0-\overline{n}||_2^6}{M^4}+\frac{1}{A}||{n}_0-\overline{n}||_2^2(||{n}_0-\overline{n}||_2^2+M^2)+\frac{1}{A}||\na n_{\neq}||_2||n_{\neq}||_2^{3}\\
\lesssim&-\frac{1}{A}\frac{||{n}_0-\overline{n}||_2^2}{M^4}(||{n}_0-\overline{n}||_2^4-M^4||{n}_0-\overline{n}||_2^2-M^6)+ \left\{\frac{1}{AB}||\na n_{\neq}||_2^2+\frac{B}{A}||n_{\neq}||_2^6 \right\} %54)
.\ea\ee
Define the following quantity $G$ to be the time integral of the terms in the $\{\cdot\}$:
\begin{align}
G(t):=&\int_0^t\frac{1}{AB}||\na n_{\neq}||_2^2+\frac{B}{A}||n_{\neq}||_2^6d\tau, \quad t\geq 0. \label{def:Gfirst}
\end{align}
By the bootstrap hypotheses, there holds $0\leq G\lesssim ||n_{in}||_2^2+||n_{in}||_{H^1}^6 A^{-1/2}\log A$.
Applying this definition,
\be\ba
\frac{d}{dt}(||n_0-\overline{n}||_2^2-G(t))\lesssim &-\frac{1}{A}\frac{||n_0-\overline{n}||_2^2}{M^4}(||n_0-\overline{n}||_2^4-M^8-M^6)\\
\lesssim&-\frac{1}{A}\frac{||n_0-\overline{n}||_2^2}{M^4}(||n_0-\overline{n}||_2^2-G(t)-\sqrt{M^8+M^6})(||n_0-\overline{n}||_2^2+\sqrt{M^8+M^6}).\\
\ea\ee
Choosing $A$ large relative to $||n_{in}||_{H^1}^6$ and universal constants, we have
\bel\label{ML2}\ba
||n_0 - \overline{n}||_2^2& \lesssim \overline{n}^2+G(t)+\sqrt{M^8+M^6} + \norm{n_{in}}_2^2 \\
& \lesssim M^2 +  ||n_{in}||_2^2 +\sqrt{M^8+M^6}\\
& =: C_{L^2}^2(||n_{in}||_2^2,M).
\ea\eel
This completes the estimate on $\norm{n_0}_2$, which implies the first estimate in conclusion (\ref{ctrl:ZeroMd}).

Next, we use (\ref{ML2}) to estimate $||\pa_y {n}_0||_2^2$. From \eqref{def:PKS} and Minkowski's integral inequality, we have for some $B > 0$,
\begin{align}
\frac{1}{2}&\frac{d}{dt}||\pa_y {n}_0||_2^2\nonumber\\
=&\lan \pa_y{n}_0,\pa_y(\frac{1}{A}\pa_{yy}n_0-\frac{1}{A}\pa_y(\pa_y {c}_0n_0)-\frac{1}{A}(\na\cdot(\na {c}_{\neq}{n}_{\neq}))_0)\ran\nonumber\\
\leq&-\frac{1}{2A}||\pa_{yy}{n}_0||_2^2+\frac{B}{A}||\pa_{yy}{c}_0{n}_0||_2^2+\frac{B}{A}||\pa_y{c}_0\pa_y{n}_0||_2^2+\frac{B}{A}||(\pa_{yy} {c}_{\neq}{n}_{\neq})_0||_{L^2(\mathbb{T})}^2+\frac{B}{A}||(\pa_y c_{\neq}\pa_y n_{\neq})_0||_{L^2(\mathbb{T})}^2\nonumber\\
\leq&-\frac{1}{2A}||\pa_{yy}{n}_0||_2^2+\frac{B}{A}||\pa_{yy}{c}_0{n}_0||_2^2+\frac{B}{A}||\pa_y{c}_0\pa_y{n}_0||_2^2+\frac{B}{A}|| n_{\neq}||_{L^2(\Torus^2)}^2||n_{\neq}||_{L^\infty(\Torus^2)}^2\nonumber\\
&+\frac{B}{A}||\pa_y c_{\neq}||_{L^\infty(\Torus^2)}^2||\pa_y n_{\neq}||_{L^2(\Torus^2)}^2.\label{pa y n 0}
\end{align}
Using (\ref{elliptic estimate appendix}) in the above estimate (\ref{pa y n 0}), we have for some $B$ (possibly adjusted from above),
\bel\label{d dt pa y n 0}\ba\frac{1}{2}\frac{d}{dt}||\pa_y {n}_0||_2^2
\leq&-\frac{1}{2A}||\pa_{yy}{n}_0||_2^2+\frac{B}{A}||\pa_{yy}{c}_0{n}_0||_2^2+\frac{B}{A}||\pa_y{c}_0\pa_y{n}_0||_2^2+\frac{B}{A}||n_{\neq}||_{L^2(\mathbb{T}^2)}^2||n_{\neq}||_{L^\infty(\mathbb{T}^2)}^2\\
&+\frac{B \CC^2}{A} ||\pa_y n_{\neq}||_{L^2(\mathbb{T}^2)}^2.
\ea\eel
Analogously to \eqref{def:Gfirst}, we define
\begin{align}
G(t):=&\int_0^t\frac{B}{A}||n_{\neq}||_{L^2(\mathbb{T}^2)}^2||{n}||_{L^\infty(\mathbb{T}^2)}^2+\frac{B\CC^2}{A}||\pa_y n_{\neq}||_{L^2(\mathbb{T}^2)}^2d\tau,\quad\forall t\in [0,T_\star]. \label{def:Gsecond}
\end{align}
By the bootstrap hypothesis (\ref{H1}),(\ref{H2}) and (\ref{H4}) and choosing $A$ large, there holds:
\be
\ba
G(t)\lesssim & \int_{0}^{T_\star}\frac{\CC^4}{A}e^{-\frac{ct}{A^{1/2}\log A}}+\frac{1}{A}\CC^{2}||\pa_y n_{\neq}||_{L^2(\mathbb{T}^2)}^2dt \lesssim \CC^{4}.
\ea
\ee
Therefore, from (\ref{d dt pa y n 0}), we have for some $B > 0$ (using also $\norm{\partial_y n_0}_2 \lesssim \norm{n_0}_2^{1/2}\norm{\partial_{yy}n_0}_2^{1/2}$),
\be\ba
\frac{d}{dt}(||\pa_yn_0||_2^2-2G(t))\leq & -\frac{||\pa_yn_0||_2^4}{ABC_{L^2}^2}+\frac{B}{A}||n_0||_2^2||\pa_yn_0||_2^2+\frac{B}{A}||n_0-\overline{n}||_2^2||\pa_yn_0||_2^2\\
\leq&-\frac{||\pa_yn_0||_2^4}{ABC_{L^2}^2}+\frac{B}{A}C_{L^2}^2||\pa_yn_0||_2^2\\
\leq&-\frac{1}{ABC_{L^2}^2}||\pa_yn_0||_2^2(||\pa_yn_0||_2^2-2G(t)-C_{L^2}^4B^2).
\ea\ee
Integrating and applying \eqref{def:Gsecond} implies the following:
\begin{align*}
||\pa_yn_0||_2^2\leq 2G(t)+C_{L^2}^4B + \norm{\partial_y n_{in}}_2^2 \lesssim  \CC^{4} + \norm{\partial_y n_{in}}_2^2. 
\end{align*}
Hence, by choosing $C_{\dot{H}^1}^2\gg  \CC^{4} + \norm{\partial_y n_{in}}_2^2$, we complete the proof of (\ref{ctrl:ZeroMd}).

\subsection{$L^\infty$ uniform control \eqref{ctrl:Linf}} \label{sec:Linfty2D}
By the bootstrap hypothesis (\ref{H2}) and (\ref{H3}), it follows that $||n||_2^2\lesssim ||n_{in}||_{H^1}^2+C_{L^2}^2(||n_{in}||_2,M) < \infty$.
As the $L^2$ norm is subcritical for 2D Patlak-Keller-Segel, it is standard (see e.g. \cite{JagerLuckhaus92,Kowalczyk05,CalvezCarrillo06} and the references therein) that this implies a uniform-in-time $L^\infty$ bound which depends only on $\norm{n}_{L^\infty(0,T_\star;L^2)}$.
Therefore, by choosing $C_\infty$ appropriately, we have \eqref{ctrl:Linf}:
\be
||n||_{L^\infty(0,T_\star;L^\infty)} \leq 2C_\infty = 2 C_\infty(||n_{in}||_{H^1}).
\ee
This completes the proof of Proposition \ref{prop:boot2D} and hence Theorem \ref{thm:2D}.

\section{Proof of Theorem \ref{thm:3D} in the case $\mathbb{T}^3$} \label{sec:3DT3}

Next we turn to the 3D case. Heuristically, we expect the problem to be effectively $L^1$ critical with critical mass $8\pi$.
As in e.g. \cite{BlanchetEJDE06}, we will need to use the free energy to obtain such a precise control.

\subsection{Basic setting and bootstrap}
Consider the Patlak-Keller-Segel equation with advection on $\mathbb{T}^3$:
\bel\label{KSC T3}
\left\{\begin{array}{rrr}\pa_t n+u(y_1)\pa_x n+\frac{1}{A}\na\cdot(\na c n)=\frac{1}{A}\de n,\\
-\de c=n-\overline{n},\\
n(\cdot,0)=n_{in}, \end{array}\right.
\eel
where $(x,y_1,y_2)\in\mathbb{T}^3$.
We use the notation
\begin{align*}
(x,y_1,y_2)\in& \mathbb{T}\times \rr^2, \\
dy=&dy_1dy_2,\\
\na_y=&(\pa_{y_1},\pa_{y_2}),\\
\de_y=&\pa_{y_1 y_1}+\pa_{y_2 y_2}.
\end{align*}
As above, the bootstrap argument is applied to prove Theorem \ref{thm:3D}. For constants $C_{ED},C_{L^2},C_{\dot{H}^1},C_{\infty}$ determined by the proof, define $T_\star$ to be the end-point of the largest interval $[0,T_\star]$ such that the following hypotheses hold for all $T \leq T_\star$:
\begin{subequations}

(1) Nonzero mode $L_t^2\dot{H}_{x,y}^1$ estimates:

\bel\ba\label{T3 H1}
\frac{1}{A}\int_0^{T_\star}||\nabla_{x,y} n_{\neq}||_{L^2(\mathbb{T}^3)}^2dt\leq& 8||n_{in}||_2^2;\\
\ea\eel

(2) Nonzero mode enhanced dissipation estimate:
\bel\label{T3 H2}\ba
||n_{\neq}||_{L^2(\mathbb{T}^3)}^2\leq& 4C_{ED}||n_{in}||_{H^1}^2e^{-\frac{ct}{A^{1/2}\log A}}, \\
\ea\eel
where $c$ is a small number independent of $A$;

(3) Zero mode time independent estimate:
\bel\ba\label{T3 H3}
||n_0||_{L^\infty_t(0,T_\star;L^2_{y})}\leq& 4C_{L^2},\\
||\pa_y n_0||_{L^\infty_t(0,T_\star;L^2_{y})}\leq& 4C_{\dot{H}^1};\\
\ea\eel

(4) $L_t^\infty L_{x,y}^\infty$ estimate of the whole solution:
\bel\label{T3 H4}\ba
||n||_{L^\infty_t(0,T_\star;L^\infty_{x,y})}\leq& 4C_\infty.
\ea\eel
\end{subequations}
As in the two-dimensional case, we introduce the following constant:
\bel
\CC:=1+M+C_{ED}^{1/2}||n_{in}||_{H^1}+C_{L^2}+C_\infty.
\eel
Here $C_{ED}$ just depends on the properties of the shear flow $u$. $C_{L^2}$ just depends on the initial data $n_{in}$, $C_{\infty}$ depends on $n_{in}$ and $C_{L^2}$, and $C_{\dot{H}^1}$ depends on $n_{in}$, $C_{L^2}$ and $C_\infty$. 
Recall that we assume that the data is initially bounded strictly away from zero from below: 
\bel\label{T3 n_{in} lower bound}
\min_{(x,y_1,y_2)\in\Torus^3} n_{in}(x,y_1,y_2)\geq {q}>0.
\eel

As in \S\ref{sec:2D}, by local well-posedness of mild solutions, the quantities on the left-hand sides of \eqref{T3 H1}, \eqref{T3 H2}, \eqref{T3 H3}, and \eqref{T3 H4} take values continuously in time. Moreover, the inequalities are all satisfied with the $4$'s replaced by $2$'s for $t$ sufficiently small.
By the standard continuation criteria for \eqref{def:PKS}, the solution exists and remains smooth on an interval $(0,t_0]$, with $t_0 > T_\star$ such that $t_0 - T_\star$ can be taken to depend only on $\norm{n(T_\star)}_{L^2}$. 
By continuity, the following proposition shows that the solution is global and satisfies the a priori estimates \textbf{(H)} for all time. 
\begin{pro} \label{prop:3DT3}
For all $n_{in}$ and $u$, if the condition (\ref{T3 n_{in} lower bound}) and the above bootstrap hypothesis (\textbf{H}) are satisfied, there exists an $A_0(\norm{n_{in}}_{L^\infty},\norm{n_{in}}_{H^1},M,q)$ such that if $A > A_0$ then the following conclusions, referred to as (\textbf{C}), hold on the interval $[0,T_\star]$:

\begin{subequations}
(1) \bel\ba\label{T3 C1}
\frac{1}{A}\int_0^{T_\star}||\nabla_{x,y} {n}_{\neq}||_2^2dt\leq& 4||n_{in}||_2^2;\\
\ea\eel

(2) \bel\label{T3 C2}\ba
||{n}_{\neq}||_2^2\leq 2C_{ED}||n_{in}||_{H^1}^2e^{-\frac{ct}{A^{1/2}\log A}};\\
\ea\eel

(3) \bel\ba\label{T3 C3}
||{n}_0||_{L^\infty_t(0,T_\star;L^2_{y})}\leq& 2C_{L^2},\\
||\pa_y {n}_0||_{L^\infty_t(0,T_\star;L^2_{y})}\leq& 2C_{\dot{H}^1};
\ea\eel

(4) \bel\label{T3 C4}\ba
||n||_{L^\infty_t(0,T_\star;L^\infty_{x,y})}\leq& 2C_\infty.
\ea\eel
\end{subequations}
\end{pro}

The main new difficulty in the 3D case arises in the proof of \eqref{3D C3}: even if non-zero modes could be neglected entirely, the evolution of $n_0$ would be given by the $L^1$ critical parabolic-elliptic Patlak-Keller-Segel.
In \cite{BlanchetEJDE06}, the free energy, together with the logarithmic Hardy-Littlewood-Sobolev inequality (see e.g. \cite{CarlenLoss92}), was applied to prove global existence up to the critical mass. 
Similarly, here we will estimate the 2D free energy of $n_0$ (no longer a conserved quantity) and apply the 2D logarithmic Hardy-Littlewood-Sobolev inequality on $n_0$. 
We are met with a small difficulty in estimating the effect of non-zero frequencies on the free energy in regions of low density; to help deal with this, we utilize a pointwise lower bound on the solution (See Lemma \ref{lem:lowbd} below).

\subsection{Estimate on the zero mode (\ref{T3 C3})}
The idea of the proof is to exploit the fact that the shear flow strongly damps the nonzero frequencies. Hence, even though the equation (\ref{KSC T3}) is posed on $\mathbb{T}^3$, we can approximate the evolution as the classical Keller-Segel equation in $\mathbb{T}^2$ with a rapidly decaying perturbation $(\na\cdot(\na c_{\neq}n_{\neq}))_0$ coming from the nonzero modes.  

First we derive an exponentially decreasing lower bound for $n$. 
\begin{lem} \label{lem:lowbd}
Under the bootstrap hypotheses \textbf{(H)} and \eqref{T3 n_{in} lower bound}, there holds the following pointwise lower bound on the solution for all $t \in [0,T^\ast]$ 
\bel\label{1/ n 0 bound}
\norm{\frac{1}{n_0(t)}}_{\infty}\leq\norm{\frac{1}{n(t)}}_\infty\leq q^{-1}e^{\frac{\overline{n}}{A}t}.
\eel
\end{lem} 
\begin{proof} 
The equation (\ref{KSC T3}) implies that at the point $(x_{\min}(t),y_{\min}(t))$ where the minimum in space of the solution is achieved, the following inequality is satisfied:
\be\ba
(\pa_t n)(x_{\min},y_{\min})=&\frac{1}{A}(\de n)(x_{\min},y_{\min})+\frac{1}{A}(n(x_{\min},y_{\min})-\overline{n})n(x_{\min},y_{\min})\\
\geq&-\frac{1}{A}\overline{n}n(x_{\min},y_{\min}),
\ea\ee
which implies that
\be
\frac{d}{dt}n_{\min}(t)\geq -\frac{1}{A}\overline{n}n_{\min}(t).
\ee
Combining this differential inequality with (\ref{T3 n_{in} lower bound}), this yields
\bel\label{n min t}
n_{\min}(t)\geq qe^{-\frac{\overline{n}}{A}t},
\eel
which completes the lemma. 
\end{proof}

Next, we study the classical 2D free energy of $n_0$ on $\mathbb{T}^2$:
\be
\cF[n_0]=\int_{\mathbb{T}^2} n_0\log n_0 -\frac{1}{2}(n_0-\overline{n})c_0 dy.
\ee
\begin{lem} 
Under the bootstrap hypotheses \textbf{(H)} and \eqref{T3 n_{in} lower bound}, for $A$ sufficiently large, there holds the following uniform bound on $t \in [0,T^\ast]$, 
\bel\label{T3 E bound}
\cF[n(t)] \leq 2\cF[n_{in}].
\eel
\end{lem} 
\begin{proof} 
By applying the hypothesis (\ref{T3 H2},\ref{T3 H4}), Minkowski's integral inequality, and (\ref{1/ n 0 bound}), the time derivative of $\cF[n_0]$ can be estimated as follows
\bel\label{T3 d dt E}\ba
\frac{d}{dt}\cF[n_0]=&-\frac{1}{A}\int n_0|\na_y \log n_0-\na_y c_0|^2dy-\frac{1}{A}\int(\na_yc_{\neq} n_{\neq})_0\cdot(\na_y\log n_0-\na_y c_0)dy\\
\leq&-\frac{1}{2A}\int n_0|\na_y \log n_0-\na_y c_0|^2dy+\frac{1}{2A}\int\frac{|(\na_yc_{\neq} n_{\neq})_0|^2}{n_0}dy\\
\leq&-\frac{1}{2A}\int n_0|\na_y \log n_0-\na_y c_0|^2dy+\frac{1}{2A}\norm{\frac{1}{n_0}}_{\infty}\norm{\na_yc_{\neq}}_{L^2(\mathbb{T}^3)}^2 \norm{n_{\neq}}_{L^\infty(\mathbb{T}^3)}^2\\
\lesssim&-\frac{1}{2A}\int n_0|\na_y \log n_0-\na_y c_0|^2dy+\frac{\CC^4}{2Aq}e^{\left(\frac{\overline{n}}{A}-\frac{c}{A^{1/2}\log A}\right)t}.
\ea\eel
Note that for $A$ sufficiently large yields: 
\bel
\int_0^\infty\frac{\CC^4}{2Aq}e^{\frac{\overline{n}}{A}t-\frac{ct}{A^{1/2}\log A}}dt\leq \int_0^\infty\frac{\CC^4}{2Aq}e^{-\frac{ct}{2A^{1/2}\log A}}dt\lesssim \frac{\CC^4}{2Aq}A^{1/2}\log A.  \label{T3 d dt E 1}
\eel
Combining (\ref{T3 d dt E}) and (\ref{T3 d dt E 1}) yields the uniform time \eqref{T3 E bound}. 
\end{proof} 

Next, we use (\ref{T3 E bound}) to get a bound on the entropy:
\begin{lem} \label{lem:nlognbdT2} 
If (\ref{T3 E bound}) holds and $A$ is chosen large enough, there exists a constant $C_{L\log L}(n_{in})$ such that 
\bel \int n_0\log^+ n_0dy \leq C_{L\log L}(n_{in}). \eel
\end{lem}
\begin{proof}
The following logarithmic Hardy-Littlewood-Sobolev inequality on a compact manifold is needed: 
\begin{thm}\cite{SW} 
Let $\mathcal{M}$ be a two-dimensional, Riemannian, compact manifold. For all $M > 0$, there exists a constant $C(M)$ such that for all non-negative functions $f \in L^1(\mathcal{M})$ such that $f \log f \in L^1$, if $\int_{\mathcal{M}} f dx = M$, then
\bel\label{log HLS T2}
\int_{\mathcal{M}}f \log fdx+\frac{2}{M}\iint_{\mathcal{M}\times\mathcal{M}} f(x) f(y) \log d(x,y) dxdy\geq -C(M),
\eel
where $d(x,y)$ is the distance on the Riemannian manifold.
\end{thm}
Let $y \in \Torus^2$ be fixed. Define the cut-off function $\varphi_y(z)\in C^\infty$ such that
\be\ba
supp(\varphi_y)=&B(y,1/4),\\
\varphi_y(z)\equiv& 1,\forall z\in B(y,1/8),\\
supp(\na \varphi_y(z))\subset& \overline{B}(y,1/4)\backslash B(y,1/8).
\ea\ee
By extending $n_0(z)$ and $c_0(z)$ periodically to $\rr^2$, we can rewrite the equation $-\de c_0=n_0-\overline{n}$ on $\mathbb{T}^2$ such that it is posed on $\rr^2$:
\be
-\de_z (\varphi_y(z)c_0(z))=(n_0(z)-\overline{n})\varphi_y(z)-2\na_z \varphi_y(z)\cdot\na_z c_0(z)-\de_z\varphi_y(z)c_0(z).
\ee
Using the fundamental solution of the Laplacian on $\rr^2$:
\be\ba
c_0(y)=&c_0(y)\varphi_y(y)\\
=&-\frac{1}{2\pi}\int_{\rr^2}\log|y-z|\bigg((n_0(z)-\overline{n})\varphi_y(z)-2\na_z \varphi_y(z)\cdot\na_z c_0(z)-\de_z\varphi_y(z)c_0(z)\bigg)dz\\
=&-\frac{1}{2\pi}\int_{|y-z|\leq \frac{1}{4}}\log|y-z|(n_0(z)-\overline{n})\varphi_y(z)dz-\frac{1}{\pi}\int_{|y-z|\leq \frac{1}{4}}\na_z\cdot(\log|y-z|\na_z \varphi_y(z)) c_0(z)dz\\
&+\frac{1}{2\pi}\int_{|y-z|\leq \frac{1}{4}}\log|y-z|\de_z\varphi_y(z)c_0(z)dz.
\ea\ee
Due to the support of $\varphi_y$, we can identify the above with an analogous integral on $\mathbb{T}^2$ with $\abs{y-z}$ replaced by $d(y,z)$.  
Therefore, we have the following estimate on the interaction energy,
\be\ba
-&\frac{1}{2}\int\displaylimits_{\mathbb{T}^2}(n_0(y)-\overline{n})c(y) dy\\
=&\frac{1}{4\pi}\iint\displaylimits_{\substack{\mathbb{T}^2\times\mathbb{T}^2\\d(y,z)\leq \frac{1}{4}}}\log d(y,z)(n_0(y)-\overline{n})(n_0(z)-\overline{n})\varphi_y(z)dzdy+\frac{1}{2\pi}\iint\displaylimits_{\substack{\mathbb{T}^2\times\mathbb{T}^2\\ \frac{1}{8}\leq d(y,z)\leq \frac{1}{4}}}(n_0(y)-\overline{n})\na_z\cdot(\log d(y,z)\na_z \varphi_y(z)) c_0(z)dzdy\\
&-\frac{1}{4\pi}\iint\displaylimits_{\substack{\mathbb{T}^2\times\mathbb{T}^2\\\frac{1}{8}\leq d(y,z)\leq \frac{1}{4}}}(n_0(y)-\overline{n})\log d(y,z)\de_z\varphi_y(z)c_0(z)dzdy\\
=&\frac{1}{4\pi}\iint\displaylimits_{d(y,z)\leq \frac{1}{8}}\log d(y,z)(n_0(y)-\overline{n})(n_0(z)-\overline{n})dzdy+\frac{1}{4\pi}\iint\displaylimits_{\frac{1}{8}\leq d(y,z)\leq \frac{1}{4}}\log d(y,z)(n_0(y)-\overline{n})(n_0(z)-\overline{n})\varphi_y(z)dzdy\\
&+\frac{1}{2\pi}\iint\displaylimits_{\frac{1}{8}\leq d(y,z)\leq \frac{1}{4}}(n_0(y)-\overline{n})\na_z\cdot(\log d(y,z)\na_z \varphi_y(z)) c_0(z)dzdy-\frac{1}{4\pi}\iint\displaylimits_{\frac{1}{8}\leq d(y,z)\leq \frac{1}{4}}(n_0(y)-\overline{n})\log d(y,z)\de_z\varphi_y(z)c_0(z)dzdy\\
=&\frac{1}{4\pi}\iint\displaylimits_{\mathbb{T}^2\times\mathbb{T}^2}\log d(y,z)n_0(y)n_0(z)dzdy-\frac{1}{4\pi}\iint\displaylimits_{d(y,z)> \frac{1}{8}}\log d(y,z)n_0(y)n_0(z)dzdy\\
&-\frac{1}{2\pi}\overline{n}\iint\displaylimits_{d(y,z)\leq \frac{1}{8}}\log d(y,z)n_0(y)dzdy+\frac{1}{4\pi}\overline{n}^2\iint\displaylimits_{d(y,z)\leq \frac{1}{8}}\log d(y,z)dzdy\\
&+\frac{1}{4\pi}\iint\displaylimits_{\frac{1}{8}\leq d(y,z)\leq \frac{1}{4}}\log d(y,z)(n_0(y)-\overline{n})(n_0(z)-\overline{n})\varphi_y(z)dzdy\\
&+\frac{1}{2\pi}\iint\displaylimits_{\frac{1}{8}\leq d(y,z)\leq \frac{1}{4}}(n_0(y)-\overline{n})\na_z\cdot(\log d(y,z)\na_z \varphi_y(z)) c_0(z)dzdy-\frac{1}{4\pi}\iint\displaylimits_{\frac{1}{8}\leq d(y,z)\leq \frac{1}{4}}(n_0(y)-\overline{n})\log d(y,z)\de_z\varphi_y(z)c_0(z)dzdy.
\ea\ee
The 2nd, 3rd, 4th, 5th terms in the last line are bounded below by $-BM^2$ for some constant $B > 0$. 
The 6th and 7th terms are bounded below by $-B M ||c_0||_{L^1}$ for some constant $B > 0$, using the fact that $\na_z\cdot(\log|y-z|\na_z \varphi_y(z))$ and $\log|y-z|\de_z\varphi_y(z)$ are bounded in the region $\frac{1}{8}\leq|y-z|\leq \frac{1}{4}$. 
Denoting $K(z)$ to be the fundamental solution of the Laplacian on $\mathbb{T}^2$, by Young's inequality, we have 
\be\ba
||c_0||_{L^1(\mathbb{T}^2)}=||K\ast (n_0-\overline{n})||_{L^{1}(\mathbb{T}^2)}\leq||K||_{L^1(\mathbb{T}^2)}||n_0-\overline{n}||_{L^{1}(\mathbb{T}^2)}\lesssim M.
\ea\ee
The calculation above hence implies the following for some constant $B>0$,
\be
 -\frac{1}{2}\int(n_0-\overline{n})(-\de)^{-1}(n_0-\overline{n}) dy\geq\frac{1}{4\pi}\iint_{\mathbb{T}^2\times\mathbb{T}^2}\log d(z,y)n_0(z)n_0(y) dzdy-BM^2.
\ee
Combining this estimate with (\ref{T3 E bound}) yields
\be\ba
 2\cF[n_{in}]\geq& \left(1- \frac{M}{8\pi}\right)\int_{\mathbb{T}^2} n_0\log n_0dy+ \frac{M}{8\pi} \left( \int_{\mathbb{T}^2} n_0\log n_0dy +\frac{2}{M}\iint_{\mathbb{T}^2\times\mathbb{T}^2}n_0(z)\log d(z,y)n_0(y) dzdy\right)-BM^2. 
\ea\ee
Applying (\ref{log HLS T2}) in the above estimate, we obtain
\be\ba
 2\cF[n_{in}]\geq&\left(1-\frac{M}{8\pi}\right)\int_{\mathbb{T}^2} n_0\log n_0dy-C(M)-BM^2, 
\ea\ee
which results in
\be\ba
\int_{\mathbb{T}^2} n_0\log n_0dy\leq \frac{ 2\cF[n_{in}]+C(M)+BM^2}{1-\frac{M}{8\pi}}. 
\ea\ee
As  $x\log x$ is bounded below, this implies the following for a suitable constant $C_{L \log L}$ depending only on the initial data due to $y \in \mathbb T^2$: 
\be
\int_{\mathbb{T}^2} n_0\log ^+n_0 dy\leq C_{L\log L}(n_{in}) < \infty. 
\ee
This completes the proof of the lemma. 
\end{proof}

\subsection{Enhanced dissipation estimate, (\ref{T3 C2})}
There are only a few differences with \S\ref{sec:ED}, which we focus on below. 
Analogous to \S\ref{sec:ED}, we define the energy $\Phi[n]$ on the torus $\Torus^3$ as follows:
\bel\label{T3 Phi k}
\Phi_k[n(t)]=||\widehat{n}_k(t)||_2^2+||\sqrt{\al}\pay \widehat{n}_k(t)||_2^2+2kRe\lan i\beta u' \widehat{n}_k(t),\pay\widehat{n}_k (t)\ran+|k|^2||\sqrt{\gamma}u' \widehat{n}_k(t)||_2^2;
\eel
\bel\label{T3 Phi}
\Phi[n(t)]=\sum_{k\neq 0} \Phi_k[n(t)] =||n_{\neq}(t)||_{2}^2+||\sqrt{\al}\pay n_{\neq}(t)||_2^2+2\lan \beta u' \pa_xn_{\neq}(t),\pay n_{\neq}
(t)\ran+||\sqrt{\gamma}u'\pa_x n_{\neq}(t)||_2^2.
\eel
Here $\alpha, \beta$, and $\gamma$ are chosen as in \eqref{def:abg}.
Analogously, we have
\begin{align}
\Phi_k[n] \approx ||\widehat{n}_k||_2^2+||\sqrt{\al}\pay \widehat{n}_k||_2^2+|k|^2||\sqrt{\gamma}u'\widehat{n}_k||_2^2,
\end{align}
and hence
\begin{align}
\norm{\widehat{n}_k}_{2}^2 + A^{-1/2}\abs{k}^{-1/2}\norm{\pay \widehat{n}_k}_{2}^2 \lesssim \Phi_k[n] \lesssim \norm{\widehat{n}_k}_{2}^2 +  \abs{k}^{1/2} A^{1/2} \norm{\widehat{n}_k}_2^2 + A^{-1/2}\abs{k}^{-1/2}\norm{\pay \wh{n}_k}_2^2. \label{T3 neq:PhiEquiv}
\end{align}
Our goal in this section is to prove the following proposition:
\begin{pro}\label{T3 C2pro}
There exists a small constant $c>0$ depending only on $u$ such that, under the bootstrap hypotheses and for $A$  sufficiently large depending only on $u$, $\norm{n_{in}}_{H^1}$ and $\norm{n_{in}}_\infty$, there holds
\bel\label{T3 Phitestimate}
\frac{d}{dt}\Phi[n(t)]\leq -\frac{c}{A^{1/2}}\Phi[n(t)].
\eel
By \eqref{ineq:PhiEquiv}, it follows that
\begin{align}
\norm{n_{\neq}}_{L^2}^2 \leq \Phi(0) e^{-cA^{-1/2}t} \lesssim A^{1/2}\norm{n_{in}}_{H^1}e^{-cA^{-1/2}t}.
\end{align}
\end{pro}
\begin{rmk}
Same as in the proof of Theorem \ref{thm:2D}, Proposition \ref{T3 C2pro} implies \eqref{T3 C2}.
\end{rmk}
On $\Torus^3$, the analogue of estimate \eqref{ddtPhik} holds.

\begin{pro}\label{T3 thm d/dt Phi k}
For $\tilde{\epsilon}$ sufficiently small depending only on $u$, there holds, 
\begin{align}
\frac{d}{dt}\Phi_k[n](t)\leq&\bigg\{-\frac{\tilde{\ep}}{2} \frac{|k|^{1/2}}{A^{1/2}}||\widehat{n}_k||_2^2-\frac{\tilde{\ep}}{2}\frac{|k|^{1/2}}{A^{1/2}}||\sqrt{\al}\pay \widehat{n}_k||_2^2-\frac{\tilde{\ep}}{2}\frac{|k|^{5/2}}{A^{1/2}}||\sqrt{\gamma}u'\widehat{n}_k||_2^2 -\frac{1}{4A}||\na_y\wh{n}_k||_2^2\nonumber\\
&-\frac{1}{2}|k|^2||\sqrt{\beta}u' \wh{n}_k||_2^2-\frac{1}{2A}|k|^2||\wh{n}_k||_2^2-\frac{1}{4A}||\sqrt{\al}\pay\na_y\wh{n}_k||_2^2\nonumber\\
&-\frac{1}{4A}|k|^4||\sqrt{\gamma}u' \wh{n}_k||_2^2-\frac{1}{4A}|k|^2||\sqrt{\gamma}u'\na_y\wh{n}_k||_2^2\bigg\}\nonumber\\
&+\bigg\{2Re\lan -L_k ,\widehat{n}_k\ran
-2Re\lan \al \pa^2_{y_1}\widehat{n}_k,-L_k\ran-2kRe[\lan i\beta u'L_k,\pay\widehat{n}_k\ran+\lan i\beta u'\widehat{n}_k,\pay L_k\ran]\nonumber\\
&+2|k|^2Re\lan \gamma (u')^2\widehat{n}_k,-L_k\ran\bigg\}\nonumber\\
&+\bigg\{-2Re\lan NL_k,\widehat{n}_k\ran
+2Re\lan  \al \pa^2_{y_1}\widehat{n}_k,NL_k\ran-2kRe[\lan i\beta u'NL_k,\pay\widehat{n}_k\ran+\lan i\beta u'\widehat{n}_k,\pay NL_k\ran]\nonumber\\
&-2|k|^2Re\lan \gamma (u')^2\widehat{n}_k,NL_k\ran\bigg\}\nonumber\\
=:&\mathcal{N}_k+\{L_k^1+L_k^{\al}+L_k^{\beta}+L_k^{\gamma}\}+\{NL_k^1+NL_k^{\al}+NL_k^{\beta}+NL_k^{\gamma}\},\label{T3 ddtPhik}
\end{align}
where $\mathcal{N}_k$ refers to the negative terms.
Recall that $L_k,NL_k$ are defined in (\ref{L},\ref{NL}).
\end{pro}
\begin{proof}
The first term in $\partial_t \Phi_k[n](t)$ is:
\bel\label{term 1}
\frac{d}{dt}||\widehat{n}_k(t)||_2^2=-\frac{2}{A}|k|^2||\widehat{n}_k ||_2^2-\frac{2}{A}||\na_y\widehat{n}_k||_2^2+2Re\lan - NL_k,\widehat{n}_k\ran+2Re\lan -L_k,\widehat{n}_k\ran.
\eel
The second term, $\frac{d}{dt}||\sqrt{\al}\pay \widehat{n}_k(t)||_2^2$, gives:
\bel\label{term 2}\ba
\frac{d}{dt}||\sqrt{\al}\pay \widehat{n}_k(t)||_2^2 =&2Re\lan \al \pay\widehat{n}_k,\pa_{y_1t}\widehat{n}_k\ran\\
=&2Re\lan \al \pay\widehat{n}_k,\pay\left(\frac{1}{A}(\de_{y}-|k|^2)\widehat{n}_k-iu(y)k\widehat{n}_k-L_k-NL_k\right)\ran\\
=&-\frac{2}{A}|k|^2||\sqrt{\al}\pay\widehat{n}_k||_2^2-\frac{2}{A}||\sqrt{\al}\pa_{y_1}\grad_y \widehat{n}_k||_2^2 -2kRe\lan \al i u'\widehat{n}_k,\pay\widehat{n}_k\ran\\
&+2Re\lan \al \pay \widehat{n}_k,\pay\left(-L_k-NL_k\right)\ran\\
=&-\frac{2}{A}|k|^2||\sqrt{\al}\pay\widehat{n}_k||_2^2-\frac{2}{A}||\sqrt{\al}\pay \grad_y \widehat{n}_k||_2^2 -2kRe\lan \al i u'\widehat{n}_k,\pay\widehat{n}_k\ran\\
&-2Re\lan \al \pa^2_{y_1}\widehat{n}_k,-L_k-NL_k\ran.\\
\ea\eel
The third term, the term involving $\beta$, can be treated as follows:
\begin{align*}
\frac{d}{dt}(2k Re\lan i\beta u' \wh{n}_k(t),\pay\wh{n}_k(t)\ran)=
&2kRe\lan i\beta u'\pa_t \wh{n}_k(t),\pay\wh{n}_k(t)\ran+2kRe\lan i\beta u'\wh{n}_k(t),\pa_{y_1 t}\wh{n}_k(t)\ran\\
=&-\frac{4k^3}{A}Re\lan i\beta u' \wh{n}_k,\pay \wh{n}_k\ran+\frac{4k}{A}Re\lan i\beta u' \pa_{y_1y_1}\wh{n}_k,\pay\wh{n}_k\ran\\
&+\frac{2k}{A}Re \lan i\beta u''' \wh{n}_k,\pay \wh{n}_k\ran-2kRe\lan i\beta u' \wh{n}_k,u'ik\wh{n}_k\ran\\
&+2kRe\lan i\beta u'(-NL_k-L_k),\pay \wh{n}_k\ran+2kRe\lan i\beta u'\wh{n}_k,\pay(-NL_k-L_k)\ran\\
&+\frac{4k}{A}Re\lan i\beta u'\pay\payy \wh{n}_k(t),\payy \wh{n}_k(t)\ran\\
\leq&-\frac{4k^3}{A}Re\lan i\beta u' \wh{n}_k,\pay \wh{n}_k\ran+\frac{4k}{A}Re\lan i\beta u' \pa_{y_1y_1}\wh{n}_k,\pay\wh{n}_k\ran\\
&+\frac{2k}{A}Re \lan i\beta u''' \wh{n}_k,\pay \wh{n}_k\ran-2|k|^2\norm{ \sqrt{\beta} u' \wh{n}_k}_2^2\\
&+2kRe\lan i\beta u'(-NL_k-L_k),\pay \wh{n}_k\ran+2kRe\lan i\beta u'\wh{n}_k,\pay(-NL_k-L_k)\ran\\
&+\frac{1}{A}||\sqrt{\al}\pay\payy \wh{n}_k||_2^2+\frac{8|k|^2\beta^2}{2A\al\gamma}||\sqrt{\gamma}u' \payy \wh{n}_k(t)||_2^2.
\end{align*}
Using that $\frac{\beta^2}{\al\gamma}\leq \frac{1}{8}$ (recall, this is ensured in \cite{BCZ15}), the corresponding terms in \eqref{term 2} and \eqref{term 4} absorb the last two terms. Other terms are treated as in \S\ref{sec:ED} and \cite{BCZ15}. 
Finally, for the term $\frac{d}{dt}|k|^2||\sqrt{\gamma}u' \widehat{n}_k(t)||_2^2$, we have
\bel\label{term 4}\ba
\frac{d}{dt}|k|^2||\sqrt{\gamma}u' \widehat{n}_k(t)||_2^2
=&-\frac{2}{A}|k|^4||\sqrt{\gamma}u'\widehat{n}_k||_2^2
-\frac{4}{A}|k|^2Re\lan \gamma u'u'' \widehat{n}_k,\pay\widehat{n}_k\ran\\
&-\frac{2}{A}|k|^2||\sqrt{\gamma}u'\grad_y \widehat{n}_k||_2^2 \\ 
&+2|k|^2Re\lan \gamma (u')^2\widehat{n}_k,-L_k-NL_k\ran.
\ea\eel
Combining the above terms yields the result.
\end{proof}

As in \S\ref{sec:ED}, the remainder of the section is devoted to controlling $L$ and $NL$ by the negative terms in (\ref{T3 ddtPhik}).

\subsubsection{Estimate on the $L$ terms in \eqref{T3 ddtPhik}}
In this section we prove that for $A$ sufficiently large,
\bel\label{T3 Lestimate}
L_k^1+L_k^\al+L_k^\beta+L_k^\gamma \leq -\frac{1}{4}\mathcal{N}_k.
\eel
We begin by estimating the $L_k^1$ term in \eqref{T3 ddtPhik}.
Using \eqref{mode by mode estimates 3D} and the bootstrap hypotheses \textbf{(H)}, we have, for any fixed constant $B \geq 1$,
\be\ba
L_k^1 =&\frac{2}{A}Re\lan \overline{n}\widehat{n}_k+2({n}_0-\overline{n})\widehat{n}_k,\widehat{n}_k\ran
-\frac{2}{A}Re\lan\na_y{c}_0\cdot\na_y\widehat{n}_k,\widehat{n}_k\ran
-\frac{2}{A}Re\lan\na_y\widehat{c}_k\cdot\na_y{n}_0,\widehat{n}_k\ran\\
\leq& \frac{2}{A}(2||{n}_0-\overline{n}||_\infty+\overline{n})||\widehat{n}_k||_2^2+\frac{1}{AB}||\na_y \widehat{n}_k||_2^2+\frac{B}{A}||\na_yc_0||_\infty^2||\widehat{n}_k||_2^2+\frac{2}{A}||n_0||_\infty||\de_y\wh{c}_k||_2||\wh{n}_k||_2\\
&+\frac{2}{A}||n_0||_\infty||\na_y\wh{c}_k||_2||\na_y \wh{n}_k||_2\\
\lesssim & \frac{1}{AB}||\na_y \widehat{n}_k||_2^2+\frac{B\CC^2}{A}||\widehat{n}_k||_2^2.
\ea\ee
Therefore, by the bootstrap hypotheses, we can choose $B$ sufficiently large, and then $A$ sufficiently large, such that the following holds:
\begin{align*}
\abs{L_k^1}  \leq -\frac{1}{16}\mathcal{N}_k,
\end{align*}
which is consistent with \eqref{T3 Lestimate}.

We turn next to $L^\al_k$ in \eqref{T3 ddtPhik}, which we divide into the following:
\bel\label{T3 L al}\ba
L^\al_k & =-2Re\lan \al \pa_{y_1}^2\widehat{n}_k,\frac{1}{A}\overline{n}\widehat{n}_k+\frac{2}{A}({n}_0-\overline{n})\widehat{n}_k-\frac{1}{A}\na_yc_0\cdot\na_y\widehat{n}_k-\frac{1}{A}\na_y\widehat{c}_k\cdot\na_y{n}_0\ran\\
& = :L^\al_{k,0}+L^\al_{k,1}+L^\al_{k,2}+L^\al_{k,3}.
\ea\eel
The treatment of the $L^\al_{k,0}$ and $L^\al_{k,1}$ terms are analogous to the treatment in \S\ref{sec:ED} and hence we omit it for the sake of brevity.
Next we estimate $L^\al_{k,2}$ in (\ref{T3 L al}).
Using \eqref{elliptic estimate appendix 3D} and the hypotheses, we have the following for any $B \geq 1$:
\be\ba
\abs{L^\al_{k,2}} \lesssim \frac{1}{BA}||\sqrt{\al}\pa_{y_1}^2\widehat{n}_k||_2^2+\frac{B}{A^{3/2}}||\na_y{c}_0||_\infty^2||\na_y\widehat{n}_k||_2^2 
\lesssim \frac{1}{BA}||\sqrt{\al}\pa_{y_1}^2\widehat{n}_k||_2^2+\frac{B\CC^2}{A^{3/2}}||\na_y\widehat{n}_k||_2^2. 
\ea\ee
Hence, by the bootstrap hypotheses and the definition of $\mathcal{N}_k$ , it follows we can choose $B$ large and then $A$ large to control this term consistent with \eqref{T3 Lestimate}.
Similarly, for $L^\al_{k,3}$ in (\ref{T3 L al}), by \eqref{mode by mode estimates 3D} and the hypothesis \eqref{T3 H3}, we have that
\be\ba
|L^\al_{k,3}| \lesssim&\frac{1}{BA}||\sqrt{\al}\pa_{y_1}^2\widehat{n}_k||_2^2+\frac{B}{A^{3/2}}||\na_y\widehat{c}_k||_\infty^2||\na_y\widehat{n}_0||_2^2\\
\lesssim& \frac{1}{BA}||\sqrt{\al}\pa_{y_1}^2\widehat{n}_k||_2^2+\frac{B}{A^{3/2}}||\widehat{n}_k||_2||\na_y\widehat{n}_k||_2||\na_y\widehat{n}_0||_2^2\\
\lesssim& \frac{1}{BA}||\sqrt{\al}\pa_{y_1}^2\widehat{n}_k||_2^2+\frac{B}{A^{3/2}}||\na_y\widehat{n}_k||_2^2+\frac{BC_{\dot{H}^1}^4}{A^{3/2}}||\widehat{n}_k||_2^2.\\
\ea\ee
As above,it follows we can choose $B$ large and then $A$ large to control this term consistent with \eqref{T3 Lestimate}.

Next, turn to the $L_k^\beta$ term in (\ref{T3 ddtPhik}), which we divide into two contributions:
\bel\label{T3 L beta}\ba
L_k^\beta=&2kRe \lan i\beta u'\widehat{n}_k,\pa_{y_1}(\frac{1}{A}\overline{n}\widehat{n}_k+\frac{2}{A}({n}_0-\overline{n})\widehat{n}_k-\frac{1}{A}\na_yc_0\cdot\na_y\widehat{n}_k-\frac{1}{A}\na_y\widehat{c}_k\cdot\na_y{n}_0)\ran  \\
&+ 2k Re \lan i\beta u'(\frac{1}{A}\overline{n}\widehat{n}_k+\frac{2}{A}({n}_0-\overline{n})\widehat{n}_k-\frac{1}{A}\na_yc_0\cdot\na_y\widehat{n}_k-\frac{1}{A}\na_y\widehat{c}_k\cdot\na_y{n}_0),\pa_{y_1}\widehat{n}_k\ran \\
=:& L^\beta_{k,1}+L^\beta_{k,2}.
\ea\eel
The first term in (\ref{T3 L beta}) is further decomposed via
\bel\label{T3 L beta 1}\ba
L^\beta_{k,1}=&2kRe\lan i\beta u'\widehat{n}_k,\pa_{y_1}\left(\frac{1}{A}\overline{n}\widehat{n}_k+\frac{2}{A}({n}_0-\overline{n})\widehat{n}_k-\frac{1}{A}\na_yc_0\cdot\na_y\widehat{n}_k-\frac{1}{A}\na_y\widehat{c}_k\cdot\na_y{n}_0\right)\ran\\
=:&L^\beta_{k,10}+L^\beta_{k,11}+L^\beta_{k,12}+L^\beta_{k,13}.
\ea\eel
The treatment of the $L^\beta_{k,10}$ and $L^\beta_{k,11}$ terms are analogous to the treatment in \S\ref{sec:ED} and are hence we omitted for the sake of brevity.
For the $L^\beta_{k,12}$ term in (\ref{T3 L beta 1}), we first estimate,
\bel\label{T3 L beta k,12}
\ba
|L^\beta_{k,12}| \leq &\bigg|2kRe\lan i\beta u'' \widehat{n}_k,\frac{1}{A}\na_y {c}_0\cdot\na_y\widehat{n}_k\ran\bigg| +\bigg|2kRe\lan i\beta u'\pa_{y_1}\widehat{n}_k,\frac{1}{A}\na_yc_0\cdot\na_{y}\widehat{n}_k\ran\bigg|\\
=:&L^\beta_{k,121}+L^\beta_{k,122}.
\ea
\eel
For $L^\beta_{k,121}$, we use \eqref{elliptic estimate appendix 3D}, the definition of $\beta$, and the bootstrap hypotheses to deduce,
\be\ba
\abs{L^\beta_{k,121}} \leq& \frac{1}{AB}||\na_y \wh{n}_k||_2^2+\frac{B|k|^2||\beta u''\wh{n}_k||_2^2||\na_y c_0||_\infty^2}{A}\\
\lesssim& \frac{1}{AB}||\na_y \wh{n}_k||_2^2+\frac{B\CC^2||\wh{n}_k||_2^2}{A}.
\ea\ee
Hence, we may choose $B$ large and then $A$ large to make these terms consistent with \eqref{T3 Lestimate}.
Next we turn to $L^\beta_{k,122}$ in (\ref{T3 L beta k,12}). Applying integration by parts,  \eqref{elliptic estimate appendix 3D}, the definition of $\beta$ and the bootstrap hypotheses, we have
\be\ba
\abs{L^\beta_{k,122}} \leq& \bigg|\frac{2k}{A} Re\lan i\beta u'\pa_{y_1}\na_y\wh{n}_k\cdot\na_y c_0,\wh{n}_k\ran\bigg|+\bigg|\frac{2k}{A}Re\langle i\beta u'\pa_{y_1}\wh{n}_k(\de c_0),\wh{n}_k\ran\bigg|\\
&+\bigg|\frac{2k}{A}Re\lan i\beta u''\pa_{y_1}\wh{n}_k,\pay c_0\wh{n}_k\ran\bigg|\\
\leq&\frac{1}{AB}||\sqrt{\al}\pay^2\widehat{n}_k||_2^2+\frac{1}{AB}||\sqrt{\al}\pay\payy\widehat{n}_k||_2^2+\frac{B|k|^2\beta}{A\alpha}||\sqrt{\beta}u' \widehat{n}_k||_2^2 ||\na_yc_0||_\infty^2\\
&+\frac{1}{AB}||\na_y \wh{n}_k||_2^2+\frac{|k|^2B}{A}||\sqrt{\beta}u'\wh{n}_k||_2^2||n_0-\overline{n}||_\infty^2\\
&+\frac{1}{AB}||\na_y \wh{n}_k||_2^2+\frac{|k|^2B}{A}||\beta u''\wh{n}_k||_2^2||\na_y c_0||_\infty^2\\
\lesssim&\frac{1}{AB}||\sqrt{\al}\na_y\pay\widehat{n}_k||_2^2+\frac{|k|^2B\CC^2}{A^{1/2}}||\sqrt{\beta}u' \widehat{n}_k||_2^2+\frac{1}{AB}||\na_y \wh{n}_k||_2^2+\frac{B\CC^2}{A}||\wh{n}_k||_2^2.
\ea\ee
Hence, we may choose $B$ large and then $A$ large to make these terms consistent with \eqref{T3 Lestimate}.
Consider next $L^\beta_{k,13}$ in (\ref{T3 L beta 1}), which we integrate by parts and further sub-divide as:
\bel\label{T3 L beta 13}\ba
L^\beta_{k,13}=&2kRe\lan i\beta u''  \widehat{n}_k+i\beta u'\pay  \widehat{n}_k,\frac{1}{A}\na_{y}\widehat{c}_k\cdot\na_y{n}_0\ran = :L^\beta_{k,131}+L^\beta_{k,132}.
\ea\eel
For $L^\beta_{k,131}$, by
(\ref{mode by mode estimates 3D}), the definition of $\beta$ and the bootstrap hypotheses, we have the following for a large constant $B \geq 1$
\be\ba
\abs{L^\beta_{k,131}} \leq&\frac{|k|^2B}{A}||\beta u'' \widehat{n}_k||_2^2||\na_y{n}_0||_2+\frac{1}{AB}||\na_y \widehat{c}_k||_\infty^2||\na_y {n}_0||_2\\
\lesssim & \frac{B||\na_y {n}_0||_2}{A}|| \widehat{n}_k||_2^2+\frac{1}{AB}||\widehat{n}_k||_2||\na_y\widehat{n}_k||_2||\na_y {n}_0||_2\\
\lesssim & \frac{C_{\dot{H}^1}B}{A}|| \widehat{n}_k||_2^2+\frac{1}{AB}||\na_y\widehat{n}_k||_2^2+\frac{C_{\dot{H}^1}^2}{A}||\widehat{n}_k||_2^2.
\ea\ee
Therefore, for $B$, then $A$, large, this term is controlled consistent with \eqref{T3 Lestimate}.

Using \eqref{mode by mode estimates 3D}, the $L^\beta_{k,132}$ term in (\ref{T3 L beta 13}) is handled as follows for a large constant $B \geq 1$:
\bel\label{T3 iii 3 I}\ba
\abs{L^\beta_{k,132}}\lesssim & \frac{|k|^{1/2}}{A^{1/2}B}||\sqrt{\al}\pay \widehat{n}_k||_2^2+\frac{B \beta^2}{A^{3/2}\alpha}|k|^{3/2} ||\na_y\widehat{c}_k||_\infty^2||\na_y{n}_0||_2^2\\
\lesssim &\frac{|k|^{1/2}}{A^{1/2}B}||\sqrt{\al}\pay \widehat{n}_k||_2^2+\frac{B}{A} \norm{\widehat{n}_k}_2\norm{\na_y\widehat{n}_k}_2||\na_y{n}_0||_2^2\\
\lesssim &\frac{|k|^{1/2}}{A^{1/2}B}||\sqrt{\al}\pay \widehat{n}_k||_2^2+\frac{1}{AB}\norm{\na_y\widehat{n}_k}_2^2+\frac{B^3C_{\dot{H}^1}^4}{A} \norm{\widehat{n}_k}_2^2.\\
\ea\eel
Therefore, by the bootstrap hypotheses, for $B$ and $A$ sufficiently large, this is consistent with \eqref{T3 Lestimate}.

Turn next to $L^\beta_{k,2}$ in (\ref{T3 L beta}), which we sub-divide as follows:
\bel\label{T3 L beta 2}\ba
L^\beta_{k,2}
=&2kRe\lan i\beta u'\frac{1}{A}\overline{n}\widehat{n}_k,\pa_{y_1}\widehat{n}_k\ran+2kRe\lan i\beta u'\frac{2}{A}({n}_0-\overline{n})\widehat{n}_k,\pa_{y_1}\widehat{n}_k\ran\\
&-2kRe\lan i\beta u'\frac{1}{A}\na_yc_0\cdot\na_y\widehat{n}_k,\pa_{y_1}\widehat{n}_k\ran-2kRe\lan i\beta u'\frac{1}{A}\na_y\widehat{c}_k\cdot\na_y{n}_0,\pa_{y_1}\widehat{n}_k\ran\\
=:&L^\beta_{k,20}+L^\beta_{k,21}+L^\beta_{k,22}+L^\beta_{k,23}.
\ea\eel
The $L^\beta_{k,22}$ term can be handled in the same manner as $L^\beta_{k,122}$.
For the $L^\beta_{k,21}$ term, we use the following straightforward estimate for a constant $B \geq 1$:
\be\ba
L^\beta_{k,21}
\lesssim &\frac{1}{AB}||\na_y\widehat{n}_k||_2^2+\frac{B\CC^2}{A}|k|^2||\sqrt{\beta}u'\widehat{n}_k||_2^2.
\ea\ee
As above, this is consistent with \eqref{T3 Lestimate} by the bootstrap hypotheses and $B$,$A$ large. The $L^\beta_{k,20}$ term is treated in the same way, so we skip the details for the sake of brevity.
The $L^\beta_{k,23}$ term can be estimated in the same manner as $L^\beta_{k,132}$ above (\ref{T3 iii 3 I}) and hence is omitted for brevity.
This completes the treatment of the $L_k^\beta$ term in \eqref{T3 ddtPhik}.

Finally, we estimate $L_k^\gamma$ in (\ref{T3 ddtPhik}).
We first sub-divide:
\bel\label{T3 L gamma}\ba
L_k^\gamma=& 2|k|^2Re\lan \gamma (u')^2\widehat{n}_k,\frac{1}{A}\overline{n}\widehat{n}_k+\frac{2}{A}({n}_0-\overline{n})\widehat{n}_k-\frac{1}{A}\na_yc_0\cdot\na_y\widehat{n}_k-\frac{1}{A}\na_y\widehat{c}_k\cdot\na_y{n}_0\ran\\
 =:&L^\gamma_{k,0}+ L^\gamma_{k,1}+L^\gamma_{k,2}+L^\gamma_{k,3}.
\ea\eel
The first and second term in (\ref{T3 L gamma}) are estimated as in \S\ref{sec:ED}; we omit the details for brevity.
For $L^\gamma_{k,2}$ in (\ref{T3 L gamma}), by \eqref{elliptic estimate appendix 3D} and the hypotheses, we have for $B \geq 1$ large,
\be\ba
L^\gamma_{k,2}
\lesssim &\frac{B\gamma }{A\beta }|k|^2||\sqrt{\beta}u'\widehat{n}_k||_2^2||\na_yc_0||_\infty^2+\frac{|k|^2}{AB}||\sqrt{\gamma}u'\na_y\widehat{n}_k||_2^2\\
\lesssim &\frac{B\CC^2}{A^{1/2}}|k|^2||\sqrt{\beta}u'\widehat{n}_k||_2^2+\frac{|k|^2}{AB}||\sqrt{\gamma}u'\na_y\widehat{n}_k||_2^2.
\ea\ee
As usual, this is consistent with \eqref{T3 Lestimate} by the bootstrap hypotheses and $B$,$A$ large.
The $L^\gamma_{k,3}$ term in (\ref{T3 L gamma}), is estimated slightly differently; using \eqref{mode by mode estimates 3D} and the hypotheses, we have for $B \geq 1$ large,
\be\ba
L^\gamma_{k,3}
\lesssim &\frac{1}{A^{1/2}B}|k|^{5/2}||\sqrt{\gamma}u'\widehat{n}_k||_2^2+\frac{B}{A^{3/2}}|k|^{3/2}\gamma ||\na_y\widehat{c}_k||_\infty^2||\na_y{n}_0||_2^2\\
\lesssim &\frac{1}{A^{1/2}B}|k|^{5/2}||\sqrt{\gamma}u'\widehat{n}_k||_2^2+\frac{B}{A}||\widehat{n}_k||_2||\na_y\widehat{n}_k||_2||\na_y{n}_0||_2^2\\
\lesssim &\frac{1}{A^{1/2}B}|k|^{5/2}||\sqrt{\gamma}u'\widehat{n}_k||_2^2+\frac{1}{AB}||\na_y\widehat{n}_k||_2^2+\frac{B^3}{A}||\widehat{n}_k||_2^2C_{\dot{H}^1}^4,
\ea\ee
this is consistent with \eqref{T3 Lestimate} by the bootstrap hypotheses and $B$,$A$ large.
This completes the proof of \eqref{T3 Lestimate}, and hence, under the bootstrap hypotheses, the contributions of the $L$ terms in \eqref{T3 ddtPhik} is absorbed by the $\mathcal{N}_k$ terms
for $A$ chosen sufficiently large.

\subsubsection{Estimate on $NL$ terms}
The treatment of these terms is essentially the same as \S\ref{sec:NLED}.
For example, for the $NL_k^1$ term in \eqref{T3 ddtPhik}, we estimate via,
\begin{align*}
-\sum_{k \neq 0}2 Re\lan NL_k,\widehat{n}_k\ran  = -\brak{\frac{2}{A}\grad \cdot \left( n_{\neq} \grad c_{\neq}\right),n_{\neq}} 
  = \frac{2}{A} \brak{n_{\neq} \grad c_{\neq},\grad n_{\neq}}    
 \leq \frac{2}{A} \norm{\grad c_{\neq}}_{\infty }\norm{\grad n_{\neq}}_{2} \norm{n_{\neq}}_{2}. 
\end{align*}
Applying \eqref{elliptic estimate appendix 3D} (together with the bootstrap hypotheses), gives the following for any constant $B > 1$,
\begin{align*}
-\sum_{k \neq 0}2 Re\lan NL_k,\widehat{n}_k\ran 
& \lesssim \frac{1}{AB}\norm{\grad n_{\neq}}_{2}^2 + \frac{B \CC^{2}}{A}\norm{n_{\neq}}_{2}^2. 
\end{align*}
By first choosing $B$ big, and then choosing $A$ large (relative to constants and $B$), these terms are absorbed by the negative terms in \eqref{T3 ddtPhik}.
As the other terms are similarly analogous, we omit the details for the sake of brevity.

\subsection{Nonzero mode $L_t^2 \dot{H}_{x,y}^1$ estimate (\ref{T3 C1})}
Computing $\frac{d}{dt}||{n}_{\neq}||_2^2$ and applying \eqref{elliptic estimate appendix 3D}, 
\bel\label{d dt n neq 2 2 T3}\ba
\frac{1}{2}\frac{d}{dt}||{n}_{\neq}||_2^2=&\lan n_{\neq},\frac{1}{A}\de n_{\neq}+\frac{1}{A}(n_0-\overline{n})n_{\neq} + \frac{1}{A}n_{\neq} {n}_0 -\frac{1}{A}\na c_0\cdot \na n_{\neq}-\frac{1}{A}\na c_{\neq}\cdot \na n_0-u(y)\pa_x n_{\neq}-\frac{1}{A}(\na\cdot(\na c_{\neq} n_{\neq}))_{\neq}\ran\\
\lesssim &-\frac{1}{2A}||\na n_{\neq}||_2^2+\frac{B}{A}||{n}_{\neq}||_2^2||\na{c}_0||_\infty^2+\frac{B}{A}||\grad c_{\neq}||_{\infty} \norm{n_{\neq}}_2||\na_y{n}_0||_2 \\
& +\frac{1}{A}||{n}_{\neq}||_2^2||{n}_0-\overline{n}||_\infty + \frac{1}{A}||{n}_{\neq}||_2^2 + \frac{1}{A}||\na n_{\neq}||_2||\na c_{\neq}||_4||n_{\neq}||_4\\
\lesssim &-\frac{1}{2A}||\na n_{\neq}||_2^2+\frac{C_{2,\infty}^2}{A}||n_{\neq}||_2^2 + \frac{C_{2,\infty} C_{\dot{H}^1}}{A} \norm{n_{\neq}}_2 +  \frac{1}{A}||\na n_{\neq}||_2||\na c_{\neq}||_4||n_{\neq}||_4.
\ea\eel
Note that, due to the bootstrap hypothesis (\ref{T3 H2}), there holds
\bel
\int_0^{T_\star} \frac{C_{2,\infty}^2}{A}||n_{\neq}||_2^2 + \frac{C_{2,\infty} C_{\dot{H}^1}}{A} \norm{n_{\neq}}_2 dt \lesssim \frac{\log A}{A^{1/2}} C_{ED}^2(1 + \norm{n_{in}}_{H^1}^2)\left(C_{2,\infty}^2 + C_{2,\infty} C_{\dot{H}^1}\right),
\eel
which can be made arbitrarily small by choosing $A$ large.
The latter term is treated via the Gagliardo-Nirenberg-Sobolev inequality, s
\bel\ba
\frac{1}{A}||\na n_{\neq}||_2||\na c_{\neq}||_4||n_{\neq}||_4 \lesssim &\frac{1}{A}||\na n_{\neq}||_2||\grad c_{\neq}||_2^{1/4}||\na^2 c_{\neq}||_2^{3/4}||n_{\neq}||_2^{1/4}||\na n_{\neq}||_2^{3/4}\\
\lesssim &\frac{1}{AB}||\na n_{\neq}||_2^2+\frac{B}{A}||n_{\neq}||_2^{10} \\
\lesssim &\frac{1}{AB}||\na n_{\neq}||_2^2+\frac{B C_{2,\infty}^8}{A}||n_{\neq}||_2^{2}.
\ea\eel
Hence, by choosing $B$ then $A$ sufficiently large, we have following $L_t^2\dot{H}_{x,y}^2$ estimate:
\be
\ba
\frac{1}{A}\int_0^{T\star}||\na n_{\neq}||_2^2dt\leq \frac{1}{A^{1/4}}+2||n_{in}||_{2}^2\leq 4||n_{in}||_{2}^2.
\ea
\ee
As a result, we have proven (\ref{T3 C1}).

\subsection{Remainder of the proof of Theorem \ref{thm:3D} in the case $\mathbb{T}^3$} 
The remaining steps in the proof of Proposition \ref{prop:3DT3} are the proofs of \eqref{T3 C3} and \eqref{T3 C4}. 
Since $L^2$ is subcritical in 3D, the proof of \eqref{T3 C4} follows as in \S\ref{sec:Linfty2D} by standard methods. 
The proof of \eqref{T3 C3} is a slightly easier variation of the arguments carried out in \S\ref{sec:zeroModeTxR2}. 
These arguments are carried out below and hence are not repeated here. This completes the proof of Proposition \ref{prop:3DT3} and hence also Theorem \ref{thm:3D} in the $\Torus^3$ case. 

\section{Proof of Theorem \ref{thm:3D} in the case $\mathbb{T}\times \rr^2$} \label{sec:3DR2}

The main difference between \S\ref{sec:3DR2} and \S\ref{sec:3DT3} is that we can no longer propagate a lower bound on the solution, which makes an estimate on the free energy more delicate.
Here, we instead use an approximate free energy for which it is easier to make estimates on the effect of low densities.

\subsection{Basic setting and bootstrap argument}
In this section, we analyse the equation:
\bel\label{KSC3}
\left\{\begin{array}{rrr}\pa_t n+u(y_1)\pa_x n+\frac{1}{A}\na\cdot(\na c n)=\frac{1}{A}\de n,\\
-\de c=n,\\
n(\cdot,0)=n_0,\end{array}\right.
\eel
in the space $\mathbb{T}\times \rr^2$. Note that the equation for $c$ is slightly different in $\Real^2$.
The basic idea behind the proof of the main theorem is the same as \S\ref{sec:3DT3}, however, we cannot use the true 2D free energy, and instead make a more complicated estimate on an approximate free energy. 
For constants $C's$ determined by the proof, define $T_\star$ to be the end-point of the largest interval $[0,T_\star]$ such that the following hypotheses hold for all $T \leq T_\star$:

\begin{subequations}
(1) Nonzero mode $L_t^2(0,T_\star;\dot{H}_{x,y}^1)$ estimate:

\bel\ba\label{3D H1}
\frac{1}{A}\int_0^{T_\star}||\nabla n_{\neq}||_{L^2(\mathbb{T}\times \rr^2)}^2dt\leq& 8||n_{in}||_2^2;\\
\ea\eel

(2) Nonzero mode enhanced dissipation estimate:
\bel\label{3D H2}\ba
||n_{\neq}||_{L^2(\mathbb{T}\times \rr^2)}^2\leq& 4C_{ED}||n_{in}||_{H^1}^2e^{-\frac{ct}{A^{1/2}\log A}}, \\
\ea\eel
where $c$ is a small number depending only on $u$ (in particular, independent of $A$);

(3) Zero mode uniform in time estimate:
\bel\ba\label{3D H3}
||n_0||_{L^\infty_t(0,T_\star;L^2_{y})}\leq& 4C_{L^2},\\
||\pa_y n_0||_{L^\infty_t(0,T_\star;L^2_{y})}\leq& 4C_{\dot{H}^1};\\
\ea\eel

(4) $L_t^\infty(0,T_\star; L_{x,y}^\infty)$ estimate of the whole solution:
\bel\label{3D H4}\ba
 ||n||_{L^\infty_t(0,T_\star;L^\infty_{x,y})} \leq& 4C_\infty.
\ea\eel
\end{subequations}
As in the two-dimensional case, we introduce the following constant:
\bel
\CC:=1+M+C_{ED}^{1/2}||n_{in}||_{H^1}+C_{L^2}+C_\infty.
\eel
The constant $C_{ED}$ depends only on $u$. The constant $C_{L^2}$ depends only on the initial data $n_{in}$. The constant $C_{\infty}$ depends on $n_{in}$ and $C_{L^2}$. Finally, the constant $C_{\dot{H}^1}$ depends on $n_{in}$, $C_{L^2}$ and $C_\infty$.

As in \S\ref{sec:2D} and \S\ref{sec:3DT3}, the proof of Theorem \ref{thm:3D} (b) is completed by the following proposition. 
\begin{pro}
For all $n_{in}$ and $u$, there exists an $A_0(\norm{n_{in}}_{H^1},||n_{in}||_\infty)$ such that if $A > A_0$ then the following conclusions, referred to as (\textbf{C}), hold on the interval $[0,T_\star]$:

\begin{subequations}
(1) \bel\ba\label{3D C1}
\frac{1}{A}\int_0^{T_\star}||\nabla_{x,y} {n}_{\neq}||_2^2dt\leq& 4||n_{in}||_2^2;\\
\ea\eel

(2) \bel\label{3D C2}\ba
||{n}_{\neq}(t)||_2^2\leq& 2C_{ED}||n_{in}||_{H^1}^2e^{-\frac{ct}{A^{1/2}\log A}};\\
\ea\eel

(3) \bel\ba\label{3D C3}
||n_0||_{L^\infty_t(0,T_\star;L^2_{y})}\leq& 2C_{L^2},\\
||\pa_y n_0||_{L^\infty_t(0,T_\star;L^2_{y})}\leq& 2C_{\dot{H}^1};\\
\ea\eel

(4) \bel\label{3D C4}\ba
 ||n||_{L^\infty_t(0,T_\star;L^\infty_{x,y})} \leq& 2C_\infty.
\ea\eel
\end{subequations}
\end{pro}

The remaining part of this section is organized as follows: in section 3.2, we prove the estimate on the zeroth mode (\ref{3D C3}); in section 3.3, we give some remark about the proof of (\ref{3D C1}),(\ref{3D C2}) and (\ref{3D C4}).
\subsection{Estimate on the zero mode (\ref{3D C3})} \label{sec:zeroModeTxR2}
For the case $y \in \mathbb R^2$, it is not clear how to estimate the contribution to $\frac{d}{dt}\cF$ from the non-zero frequencies at small values of $n_0$. 
The idea is to find a new (approximate) free energy which is better adapted. 
We use the following as the new approximate free energy: 
\begin{align}
\cF_{\Gamma}[n_0]=\int n_0\Gamma(n_0)-\frac{n_0c_0}{2}dy,\label{def:En0}
\end{align}
where $\Gamma $ is defined as
\bel
\Gamma(n_0)=\left\{\begin{array}{rr}\log n_0 , n_0\geq 1,\\(n_0-1)-\frac{(n_0-1)^2}{2},n_0<1.\end{array}\right.
\eel
The $\Gamma$ function is chosen such that it matches $\log$ when $n_0$ is large but is bounded from below when $n_0$ is small. Here, we have replaced the function $\log (1+(n_0-1))$ by its degree two Taylor expansion centered at 1 when $n_0<1$ and use the original $\log$ function when $n_0\geq 1$.

Next, we apply a sequence of lemmas to prove that under the bootstrap hypothesis (\ref{3D H1}), (\ref{3D H2}) and (\ref{3D H4}), the conclusion on the zero-mode (\ref{3D C3}) is true given $A$ sufficiently large.
The first is the following.
\begin{lem}\label{Lemma:free energy evolution} The time derivative of the approximate free energy $\cF_{\Gamma}[n_0]$, defined in \eqref{def:En0}, satisfies the following estimate:
\begin{equation}\label{free energy evolution}
\frac{d}{dt}\cF_{\Gamma}[n_0(t)] \lesssim \frac{1}{A}||(\na_y c_{\neq}n_{\neq})_0||_{L^2(\rr^2)}^2+\frac{1}{A}||(\na_y c_{\neq}n_{\neq})_0||_{L^{4/3}(\rr^2)}||n_0||_{L^{4/3}(\rr^2)}.
\end{equation}
Furthermore, the following quantity is bounded:
\bel\label{Gamma- control}
-\int_{n_0<1}n_0\Gamma(n_0) dy \leq \frac{3}{2}M.
\eel
\end{lem}
\begin{proof}
Taking the time derivative of $\cF_{\Gamma}[n_0(t)]$ yields
\begin{align}
\frac{d}{dt}\left(\int n_0\Gamma( n_0)-\frac{n_0c_0}{2}dy\right)
=&\int (n_0)_t(\Gamma (n_0)- c_0)dy+\int n_0(\Gamma (n_0))_tdy\nonumber\\
=&\frac{1}{A}\bigg\{-\int (n_0\na_y \log n_0-\na_y c_0n_0)\cdot(\ga '(n_0)\na_y n_0-\na_y c_0)dy\nonumber\\
&-\int\na_y(n_0\ga'(n_0))\cdot(n_0\na_y \log n_0 -\na_y c_0 n_0)dy\bigg\}\nonumber\\
&+\frac{1}{A}\bigg\{\int (\na_y c_{\neq }n_{\neq})_0\cdot(\ga'(n_0)\na_y n_0-\na_y c_0)dy+\int \na_y(n_0\ga'(n_0))\cdot(\na_y c_{\neq }n_{\neq})_0dy\bigg\}\nonumber\\
=:&\frac{1}{A}T_0+\frac{1}{A}T_{\neq}.\label{T 0, T neq}
\end{align}
The proof of the lemma is completed once we show that the term $T_0$ is non-positive and the term $T_{\neq}$ is controlled in an appropriate way.
Using the definition of $\Gamma$, we have,
\be\ba
T_0
=&-\int_{n_0\geq 1}(n_0\na_y \log n_0-\na_y c_0 n_0)\cdot(\frac{1}{n_0}\na_y n_0-\na_y c_0)dy \\ &  -\int_{n_0<1}(n_0\na_y \log n_0-\na_y c_0 n_0)\cdot((2-n_0)\na_y n_0-\na_y c_0)dy\\
& -\int_{n_0<1}(2-n_0)\na_y n_0\cdot(n_0\na_y \log n_0-\na_y c_0 n_0)dy + \int_{n_0<1}n_0\na_y n_0 \cdot(n_0\na_y \log n_0-\na_y c_0 n_0)dy.\ea\ee
Notice the following inequality:
\be
\sup_{n_0<1}\sqrt{(-3n_0+4)n_0}\leq \frac{2}{\sqrt{3}}<2,
\ee
which implies,
\be\ba
T_0=&-\int_{n_0\geq 1}n_0|\na_y \log n_0-\na_y c_0|^2dy-\int_{n_0<1}(4-3n_0)|\na_y n_0|^2dy\\
&+\int_{n_0<1}\sqrt{(-3n_0+4)n_0}\sqrt{(-3n_0+4)n_0}\na_y c_0 \cdot \na_y n_0dy-\int_{n_0<1}n_0|\na_y c_0|^2 dy+\int_{n_0<1} \na_y n_0\cdot\na_y c_0 dy\\
\leq &-\int_{n_0\geq 1}n_0|\na_y \log n_0-\na_y c_0|^2dy-\int_{n_0<1}(4-3n_0)|\na_y n_0|^2dy\\
&+\frac{2}{\sqrt{3}}\int_{n_0<1}\sqrt{(-3n_0+4)n_0}|\na_y c_0||\na_y n_0|dy-\int_{n_0<1}n_0|\na_y c_0|^2 dy+\int_{n_0<1} \na_y n_0\cdot\na_y c_0 dy.\\
\ea\ee
Completing a square using the 2nd, 3rd, 4th terms in the last line yields
\bel\label{3D zero mode T 0 1}\ba
T_0\leq &-\int_{n_0\geq 1}n_0|\na_y \log n_0-\na_y c_0|^2dy-\frac{2}{3}\int_{n_0<1}(4-3n_0)|\na_y n_0|^2dy\\
&-\int_{n_0 <1} \left(\sqrt{4-3n_0}\frac{1}{\sqrt{3}}|\na_y n_0|-\sqrt{n_0}|\na_y c_0|\right)^2dy+\int_{n_0<1} \na_y n_0\cdot\na_y c_0 dy.
\ea\eel
Now the key is to prove that the term $\int_{n_0<1} \na n_0\cdot\na c_0 dy$ in (\ref{3D zero mode T 0 1}) is negative. We want to integrate by parts to use the relation $-\Delta c_0 = n_0$ and the divergence theorem, however, the level set of a smooth function is not necessarily a smooth sub-manifold.
We recall the Sard theorem and the inverse image theorem of differential topology:
\begin{thm}(Sternberg (\cite{Sternberg64}, Theorem II.3.1); Sard \cite{Sard}) Let $f: \mathbb{R}^n \rightarrow \mathbb{R}^m$ be $C^k$, (that is, k times continuously differentiable), where $k\geq \max\{n-m+1, 1\}$. Let X denote the critical set of f, which is the set of points $x\in \mathbb{R}^n$ at which the Jacobian matrix of f has rank $< m$. Then the image f(X) has Lebesgue measure 0 in $\mathbb{R}^m$. (In other words, almost all points in the image is a regular value.)
\end{thm}
\begin{thm}(\cite{Hirsch76}, page 14)
Let $W$ in $\rr^n$ be an open set and $f:W\rightarrow \rr^q$ a $C^r$ map, $1\leq r\leq \infty$. Suppose $y\in f(W)$ is a regular value of f; this means that $f$ has rank $q$ at every point of $f^{-1}(y)$. (Therefore $q\leq n$) Then the subset $f^{-1}(y)$ is a $C^r$ submanifold of $\rr^n$ of codimension $q$.
\end{thm}
Since the solution $n(t)$ is $C^\infty$ for $t \in (0,T_\star]$, if one is not a regular value for $n_0$, by Sard's theorem and the inverse image theorem, we may find a sequence of $K_j$ such that $K_j \nearrow 1$
and that the level set $\{n_0=K_j\}$ smooth. Therefore, we may integrate by parts:
\be\ba
\int_{n_0\leq K_j<1}\na_y n_0\cdot\na_y c_0 dy=&-\int_{n_0\leq K_j}n_0 \de_y c_0 dy +\int_{n_0=K_j}n_0\na_y c_0\cdot \nu ds\\
=&\int_{n_0\leq K_j}n_0^2dy+K_j \int_{n_0=K_j}\na_y c_0\cdot \nu ds\\
=&\int_{n_0\leq K_j}n_0^2dy+K_j\int_{n_0\leq K_j}\de c_0 dy\\
\leq&\int_{n_0\leq K_j}n_0^2 dy-K_j \int_{n_0 \leq K_j} n_0dy\\
\leq &0.
\ea\ee
As we have $|\na_y n_0\cdot\na_y c_0|\in L^1$ by the Lebesgue dominated convergence theorem we deduce that $\int_{n_0<1}\na_y n_0\cdot\na_y c_0dy\leq 0$.
Therefore, we deduce from (\ref{3D zero mode T 0 1}) that,
\bel\label{3D zero mode T 0}
T_0\leq -\int_{n_0\geq 1}n_0|\na_y \log n_0-\na_y c_0|^2dy-\frac{2}{3}\int_{n_0<1}(4-3n_0)|\na_y n_0|^2dy.
\eel
This finishes the treatment of $T_0$ in (\ref{T 0, T neq}).

Now we come to the treatment of the $T_{\neq}$ in (\ref{T 0, T neq}). The idea is to use the negative terms in (\ref{3D zero mode T 0}) and the fast decay from the bootstrap hypotheses to control part of the influence from $T_{\neq}$. By Young's inequality, we estimate $T_{\neq}$ as follows:
\bel\label{3D zero mode T neq}\ba
T_{\neq} =&\int_{n_0\geq 1}(\na_y c_{\neq}n_{\neq})_0\cdot(\na_y \log n_0-\na_y c_0)dy+\int_{n_0<1}(\na_y c_{\neq}n_{\neq})_0 \cdot ((2-n_0)\na_y n_0-\na_y c_0)dy\\
&+\int_{n_0<1}(2-n_0)\na_y n_0\cdot (\na_y c_{\neq}n_{\neq})_0dy-\int_{n_0<1}n_0\na_y n_0\cdot (\na_y c_{\neq}n_{\neq})_0dy\\
\leq&\int_{n_0\geq 1}\frac{|(\na_y c_{\neq}n_{\neq})_0|}{\sqrt{n_0}}\sqrt{n_0}|\na \log n_0-\na c_0|dy+\int_{n_0<1 }(\na_y c_{\neq}n_{\neq})_0\cdot((4-3n_0)\na_y n_0)dy  \\
&+\int_{n_0<1}|(\na_y c_{\neq}n_{\neq})_0 \cdot \na_y c_0|dy\\
\leq &\int_{n_0\geq 1}B|(\na_y c_{\neq}n_{\neq})_0|^2dy+\frac{1}{B}\int_{n_0\geq 1}n_0|\na_y \log n_0-\na_y c_0|^2dy+\int_{n_0< 1}B|(\na_y c_{\neq}n_{\neq})_0|^2(4-3n_0)dy\\
&+\frac{1}{B}\int_{n_0<1 }(4-3n_0)|\na_y n_0|^2dy+||(\na_y c_{\neq}n_{\neq})_0||_{L^{4/3}(\rr^2)}||\na_y c_0||_{L^4(\rr^2)}.
\ea\eel
Finally, the $||\na_y c_0||_{L^4(\rr^2)}$ in the last line is estimated using the Hardy-Littlewood-Sobolev inequality:
\bel\label{3D zero mode T neq 1}
||\na_y c_0||_{L^4(\rr^2)}\lesssim||n_0||_{L^{4/3}(\rr^2)}.
\eel
Combining (\ref{T 0, T neq}), (\ref{3D zero mode T 0}), (\ref{3D zero mode T neq}) and (\ref{3D zero mode T neq 1}) yields (\ref{free energy evolution}). Estimate (\ref{Gamma- control}) follows from the fact that the function $\Gamma$ is bounded from below. This finishes the proof of Lemma \ref{Lemma:free energy evolution}.
\end{proof}

\begin{lem}\label{Lemma:n log^+n bound}
For $A$ sufficiently large, the approximate free energy is bounded via the following
\begin{align}
\cF_{\Gamma}[n_0(t)] \leq 2\cF_{\Gamma}[n_0(0)]. \label{ineq:cFbd}
\end{align}
 Moreover, there is a constant $C_{L \log L}$ which depends only on $\cF_{\Gamma}[n_0(0)]$ and $M$ such that the following holds independent of time:
\bel\label{n log^+n bound}
\int n_0\log^+n_0 dy\leq C_{L\log L}(\cF_{\Gamma}[n_0(0)],M).
\eel
\end{lem}
\begin{proof}
By an argument similar to the 2D case, we have that $||n_0||_{L^1(\mathbb{T}\times\rr^2)} = M$ is preserved.
Next, we estimate the right-hand side of \eqref{free energy evolution}. 

By Minkowski's inequality,
and applying the elliptic estimate (\ref{elliptic estimate appendix 3D}), we have 
\bel\label{LlogL 1}\ba
\frac{1}{A}||(\na_y c_{\neq}n_{\neq})_0||_{L^2(\rr^2)}^2\leq&\frac{1}{A}||\na_y c_{\neq}n_{\neq}||_{L^2(\mathbb{T}\times\rr^2)}^2\\
\leq&\frac{1}{A}||\na_yc_{\neq}||_\infty^2||n_{\neq}||_{2}^2\\
\lesssim & \frac{\CC^2}{A} \norm{n_{\neq}}_2^2.
\ea\eel
By \eqref{H2}, the time integral of this contribution can be made arbitrarily small by choosing $A$ sufficiently large.
Next, we estimate the second term on the right hand side of (\ref{free energy evolution}). Combining Minkowski's integral inequality, H\"{o}lder's inequality, and the Gagliardo-Nirenberg-Sobolev inequality, we obtain
\bel\label{LlogL 2}\ba
\frac{1}{A}||n_0||_{L^{4/3}(\rr^2)}||(\na_y c_{\neq}n_{\neq})_0||_{L^{4/3}(\rr^2)}
\leq&\frac{1}{A}||n_0||_{L^{4/3}(\mathbb{T}\times\rr^2)}||\na_y c_{\neq}n_{\neq}||_{L^{4/3}(\mathbb{T}\times\rr^2)}\\
\leq &\frac{1}{A}||n_0||_1^{1/2}||n_0||_2^{1/2}||\na_y c_{\neq}||_2||n_{\neq}||_4\\
\lesssim&\frac{1}{A}||n_0||_1^{1/2}||n_0||_2^{1/2}||n_{\neq}||_2||n_{\neq}||_2^{1/4}||\na n_{\neq}||_2^{3/4}\\
\lesssim&\frac{\CC^2}{A^{3/4}}||n_{\neq}||_{2}^2+\frac{1}{A^{5/4}}||\na n_{\neq}||_{2}^2.
\ea\eel
Plugging the estimates (\ref{LlogL 1}) and (\ref{LlogL 2}) in (\ref{free energy evolution}) yields
\bel\label{LlogL 3}
\frac{d}{dt}\int n_0\Gamma(n_0)-\frac{n_0c_0}{2}dy
\lesssim \frac{\CC^2}{A^{3/4}}\norm{n_{\neq}}_2^2 + \frac{1}{A^{5/4}}\norm{\grad n_{\neq}}_2^2.
\eel
It follows from the hypotheses (\ref{3D H2}),(\ref{3D H3}) and (\ref{3D H4}), that the time integral of the right hand side of (\ref{LlogL 3}) can be made arbitrarily small by choosing $A$ large relative to the quantities $||n_{in}||_2,||n_{in}||_{H^1},C_{2,\infty}$, and hence \eqref{ineq:cFbd} follows.   

A uniform in time bound on the free energy \eqref{ineq:cFbd} can be translated to a uniform in time bound on the entropy (\ref{n log^+n bound}) by using the logarithmic Hardy-Littlewood-Sobolev inequality, which we recall here (see e.g. \cite{CarlenLoss92}):

\begin{thm}[Logarithmic Hardy-Littlewood-Sobolev Inequality]For all $M > 0$, there exists a constant $C(M)$ such that for all be a nonnegative functions in $f \in L^1(\rr^2)$ such that $f\log f$ and $f\log(1+|x|^2)$ belong to $L^1(\rr^2)$. If $\int_{\rr^2}fdx=M$, then
\bel\label{log HLS}
\int_{\rr^2}f\log f dx+\frac{2}{M}\iint_{\rr^2\times\rr^2} f(x)f(y)\log|x-y|dxdy\geq -C(M). 
\eel
\end{thm}
We rewrite the approximate free energy so that the inequality (\ref{log HLS}) can be applied, \eqref{ineq:cFbd}:
\be\ba
 2\cF_{\Gamma}[n_{in}]
\geq &\int n_0\Gamma(n_0)-\int\frac{n_0c_0}{2}dy\\
=&\int n_0\log^+n_0dy+\int_{n_0\leq 1}n_0\Gamma(n_0)dy+\frac{1}{4\pi}\iint \log|w-y|n_0(w)n_0(y)dwdy\\
=&\left(1-\frac{M}{8\pi}\right)\int n_0\log^+n_0dy+\int_{n_0\leq 1}n_0\Gamma(n_0)dy\\
&+ \frac{M}{8\pi}\left(\int n_0\log^+n_0dy+\frac{2}{M}\iint \log|w-y|n_0(w)n_0(y)dwdy \right).
\ea\ee
Applying the log-HLS (\ref{log HLS}) and (\ref{Gamma- control}) yield:
\be\ba
 2\cF_{\Gamma}[n_{in}]\geq& \left(1-\frac{M}{8\pi}\right)\int n_0\log^+n_0dy+\int_{n_0\leq 1}n_0\Gamma(n_0)dy- C(M)\frac{M}{8\pi}\\
\geq& \left(1-\frac{M}{8\pi}\right)\int n_0\log^+n_0dy-\frac{3}{2}M-C(M) \frac{M}{8\pi},
\ea\ee
which leads to a bound on the entropy
\be
\int n_0\log^+n_0dy\leq\frac{8\pi}{8\pi-M}\left( 2 \cF[n_{in}]+\frac{3}{2}M+C(M)\frac{M}{8\pi}\right).
\ee
This concludes the proof of Lemma \ref{Lemma:n log^+n bound}.
\end{proof}

\begin{lem}
The bound on the entropy (\ref{n log^+n bound}) yields a uniform in time $L^2$ bound of $n_0$, that is,
\bel\label{L^2 control of n_0}
||n_0||_{L^2}\leq C_{L^2}(n_{in}).
\eel
\end{lem}
\begin{proof}
The proof is a small variation of classical Patlak-Keller-Segel techniques (see e.g. \cite{JagerLuckhaus92,BlanchetEJDE06}); we sketch the proof here for completeness.

Let $K>1$ be a constant, to be chosen later.
Observe that \eqref{n log^+n bound}, 
\bel\label{pick K here}
\int(n_0-K)_+dy\leq\int_{n_0>K}n_0dy\leq\frac{1}{\log(K)}\int_{n_0>K}n_0\log^+ (n_0)dy\leq \frac{C_{L\log L}}{\log(K)}.
\eel
Next, via \eqref{def:PKS},
\begin{align}
\frac{1}{2}&\frac{d}{dt}\int (n_0-K)_+^2dy\nonumber\\
=&\frac{1}{A}\int(n_0-K)_+[\de_y n_0-\na_y\cdot(n_0 \na c_0)-\na_y\cdot(\na_y c_{\neq}n_{\neq})_0]dy\nonumber\\
=&-\frac{1}{A}\int |\na((n_0-K)_+)|^2dy+\frac{1}{2A}\int(n_0-K)^3_+dy+\frac{3K}{2A}\int(n_0-K)_+^2dy +\frac{K^2}{A}\int(n-K)_+
dy\nonumber\\
&+\frac{1}{A}\int \na(n_0-K)_+\cdot (\na_y c_{\neq}n_{\neq})_0dy\nonumber\\
\leq&-\frac{7}{8A}\int |\na((n_0-K)_+)|^2dy+\frac{1}{2A}\int(n_0-K)^3_+dy+\frac{3K}{2A}\int(n_0-K)_+^2dy +\frac{K^2M}{A}\nonumber\\
&+\frac{B}{A}||(\na_y c_{\neq}n_{\neq})_0||_{L^2(\rr^2)}^2.\label{d dt n 0-K 2}
\end{align}
Starting with the second term in (\ref{d dt n 0-K 2}), applying the Gagliardo-Nirenberg-Sobolev inequality yields (see e.g. \cite{BlanchetEJDE06} and the references therein)
\be
\int |(n_0-K)_+|^3dy\lesssim \int|\na(n_0-K)_+|^2dy\int (n_0-K)_+dy.
\ee
From (\ref{pick K here}) that we can choose $K$ depending only on $C_{L \log L}$ such that:
\bel\label{d dt n 0-K 2 1}
-\frac{7}{8A}\int |\na((n_0-K)_+)|^2dy+\frac{1}{2A}\int(n_0-K)^3_+dy\leq -\frac{1}{2A}\int |\na((n_0-K)_+)|^2dy.
\eel
Next, we apply Minkowski's inequality, the elliptic estimate (\ref{elliptic estimate appendix 3D}), and the hypothesis (\ref{3D H4}) to control the non-zero mode contribution $||(\na_y c_{\neq}n_{\neq})_0||_{L^2(\rr^2)}^2$ in  (\ref{d dt n 0-K 2}):
\begin{align}
\frac{1}{A}||(\na_y c_{\neq}n_{\neq})_0||_{L^2(\rr^2)}^2 
\lesssim \frac{1}{A}||\na_y c_{\neq}||_\infty^2||n_{\neq}||_2^2 %\nonumber\\
\lesssim \frac{1}{A}\CC^2 ||n_{\neq}||_{2}^2.  \label{d dt n 0-K 2 2}
\end{align}
Plugging (\ref{d dt n 0-K 2 1}) and (\ref{d dt n 0-K 2 2}) into (\ref{d dt n 0-K 2}) yields
\begin{align}
\frac{1}{2}\frac{d}{dt}\int (n_0-K)_+^2dy
\lesssim& -\frac{1}{2A}\int |\na((n_0-K)_+)|^2dy+\frac{3K}{2A}\int(n_0-K)_+^2dy\nonumber\\
&+\frac{K^2M}{A}+\frac{\CC^2}{A}||n_{\neq}||_{2}^2.\label{d dt n 0-K 2 2.5}
\end{align}
Applying the Nash inequality
\be
||v||_{L^2(\rr^2)}^2\lesssim ||\na v||_{L^2(\rr^2)}||v||_{L^1(\rr^2)}
\ee
in the estimate (\ref{d dt n 0-K 2 2.5}), we obtain
\bel\label{d dt n 0-K 2 3}\ba
\frac{1}{2}\frac{d}{dt}|| (n_0-K)_+||_2^2 \lesssim -\frac{1}{2A}\frac{||(n_0-K)_+||_2^4}{M^2}+\frac{3K}{2A}||(n_0-K)_+||_2^2 +\frac{K^2M}{A}+\frac{1}{A}\CC^2\norm{n_{\neq}(t)}_{2}^2.
\ea\eel
Applying an argument similar to the one used in Section \ref{G(t)trick} to deduce \eqref{ML2}, by choosing $A$ sufficiently large implies $\int(n-K)_+^2dy\leq C(n_{in})$.
Recall the following classical inequality (see e.g. \cite{JagerLuckhaus92,CalvezCarrillo06})
\begin{align*}
\norm{n_0}_{L^2} & \lesssim \norm{(n_0 - K)_+}_{L^2} + K^{1/2} M^{1/2},
\end{align*}
where the implicit constant is independent of $K$ and $M$.
The inequality (\ref{L^2 control of n_0}) hence follows.
\end{proof}

Next, we prove the  higher regularity estimate (\ref{3D H3}) using (\ref{L^2 control of n_0}).
\begin{lem}\label{Lemma:H^1 control of n_0}
For $A$ sufficiently large, provided (\ref{L^2 control of n_0}) holds,  the following improvement to (\ref{3D H3}) holds on $[0,T_\star]$ for a suitable choice of $C_{\dot{H}^1}$:
\bel\label{H^1 control of n_0}
||\na_y n_0||_{L^2(\rr^2)} \leq 2C_{\dot{H}^1}.
\eel
\end{lem}
\begin{proof}
We employ the following standard multi-index notation:
\be\ba
\alpha=&(\alpha_1,\alpha_2)\in \mathbb{N}^2,\\
\pa_y^\alpha=&\pa_{y_1}^{\alpha_1}\pa_{y_2}^{\alpha_2},\\
||\pa_y^s n_0||_2^2=&\sum_{|\alpha|=s}||\pa_y^\alpha n_0||_2^2.
\ea\ee
Let $\alpha$ be such that $\abs{\alpha} = 1$.
Computing the time derivative of $||\pa_y^\al n_0||_2^2$ and applying $\ep$-Young's inequality:
\bel\label{pa y al n 0}\ba
\frac{1}{2}\frac{d}{dt}||\pa_y^\al n_0||_2^2
=&-\frac{1}{A}\int |\pa_y^\alpha\na n_0|^2dy+\frac{1}{A}\int \pa_y^\alpha\na n_0\cdot\pa_y^\alpha(\na c_0n_0)dy
+\frac{1}{A}\int \pa_y^\al \na n_0\cdot \pa_y^\al(\na c_{\neq} n_{\neq})_0dy\\
=&-\frac{1}{A}\int |\pa_y^\alpha\na n_0|^2dy+\frac{1}{A}\int \pa_y^\al\na n_0\cdot (\pa_y^\al\na c_0 n_0)dy\\
&+\frac{1}{A}\int \pa_y^\al\na n_0\cdot\na c_0\pa_y^{\alpha}n_0dy+\frac{1}{A}\int \pa_y^\al \na n_0\cdot \pa_y^\al(\na c_{\neq} n_{\neq})_0dy\\
\leq&-\frac{7}{8A}\int |\pa_y^\al\na n_0|^2dy+\frac{B}{A}||\pa_y^\al\na c_0n_0||_2^2+\frac{B}{A}||\na c_0 \pa_y^\al n_0||_2^2\\
&+\frac{B}{A}||\pa_y^\al(\na c_{\neq} n_{\neq})_0||_2^2\\
=:&-\frac{7}{8A}\int |\pa_y^\al\na n_0|^2dy+T_1+T_2+NZ.
\ea\eel
We first estimate the term $T_1$ in (\ref{pa y al n 0}). Combining the bound (\ref{L^2 control of n_0}), the Gagliardo-Nirenberg-Sobolev inequality, and the $L^4$ boundedness of the Riesz transform yields,
\bel\label{T_1}\ba
T_1 \lesssim &\frac{1}{A} \norm{\pa_{y}^\al\grad c_0}_4^2||n_0||_4^2
\lesssim \frac{1}{A} ||n_0||_4^4 
\lesssim \frac{1}{A}||n_0||_2^2||\na_y n_0||_2^2
\lesssim \frac{1}{A}C_{L^2}^2||\na_y n_0||_2^2.
\ea\eel
Next for the second term $T_2$ in (\ref{pa y al n 0}), combining the elliptic estimate (\ref{elliptic estimate appendix 3D}) and the hypothesis (\ref{3D H4}) yields
\bel\label{T_2}\ba
T_2
\leq\frac{B}{A}||\na_y c_0||_{\infty}^2||\pa_y^\al n_0 ||_2^2
\lesssim \frac{1}{A}\CC^{2}||\na_y n_0||_2^2.
\ea\eel
Similar to the two dimensional case, the $NZ$ term in (\ref{pa y al n 0}) is estimated using Minkowski's inequality, the elliptic estimate (\ref{elliptic estimate appendix 3D}), (\ref{3D H2}), and (\ref{3D H4}) as follows:
\bel\label{NZ}\ba
NZ\leq&\frac{B}{A}||n_{\neq}||_2^2||n_{\neq}||_\infty^2+\frac{B}{A}||\na_y c_{\neq}||_\infty^2||\na_y n_{\neq}||_2^2\\
\lesssim &\frac{B \CC^2}{A}||n_{\neq}||_2^2+\frac{B \CC^2}{A} ||\na_y n_{\neq}||_2^2.
\ea\eel
Combining the the above estimates (\ref{pa y al n 0}),(\ref{T_1}),(\ref{T_2}),(\ref{NZ}) and summing over $\al=1,2$ yield
\bel\label{energy estimate zero mode 2}\ba
\frac{1}{2}\frac{d}{dt}||\na_y n_0||_2^2
\leq -\frac{1}{2A}||\pa_y^{2}n_0||_2^2+\frac{B}{A}\CC^2||\na_y n_0||_2^2+G'(t),
\ea\eel
where $G(t)$ is defined as
\be
G(t):=\int_0^t\frac{B\CC^2}{A}||n_{\neq}||_2^2 +\frac{B \CC^2}{A}||\na_y n_{\neq}||_2^2d\tau,\quad\forall t\in[0,T_\star].
\ee
Applying an argument similar to the one used in Section \ref{G(t)trick} to prove \eqref{ML2}, choosing $A$ sufficiently large implies that
\bel
\label{zero mode control dim 3}
||\na_y n_0||_{2}^2\lesssim\CC^4
\eel
which is independent of $A$ and $C_{\dot{H}^1}$. Note that we still have the freedom to pick our $C_{\dot{H}^1}$, and we choose it such that $C_{\dot{H}^1}^2$ is much bigger than the right hand side of (\ref{zero mode control dim 3}). This finishes the proof of Lemma \ref{Lemma:H^1 control of n_0} and the conclusion (\ref{3D C3}) follows.
\end{proof}

\appendix
\section{Appendix}
\subsection{$\na c$ estimates}
We have applied various estimates on $\na c_0,\na c_{\neq}$; while all are standard, we sketch the proofs here for the readers' convenience. 

\begin{lem}\label{lem 2}In the two-dimensional case, the following estimate holds for uniformly for all $k \in \Integers \setminus \set{0}$ and $(k^2-\pa_{yy})\wh{c}_k=\wh{n}_k$:
\bel
|k|^{1/2}||\pa_y \widehat{c}_k||_{L^\infty(\mathbb{T})}\lesssim ||\widehat{n}_k||_{L^2(\mathbb{T})}.
\eel
\end{lem}
\begin{proof} By the Gagliardo-Nirenberg-Sobolev inequality,  
\be
|k|^{1/2}||\pa_y \widehat{c}_k||_\infty\lesssim |k|^{1/2}||\widehat{c}_k||_2^{1/4}||\pa_{yy}\widehat{c}_k||_2^{3/4}\lesssim|||k|^2 \widehat{c}_k||_2^{1/4}||\pa_{yy}\widehat{c}_k||_2^{3/4}\lesssim||\widehat{n}_k||_2.
\ee
This completes the proof of the lemma.
\end{proof}

\begin{lem}\label{lem 3}In the two-dimensional case, the following estimate on $\na c_0$ holds:
\bel
||\pa_y {c}_0||_{L^\infty(\mathbb{T})}\lesssim ||{n}_0-\overline{n}||_{L^1(\mathbb{T})}.
\eel
\end{lem}
\begin{proof}
By the fundamental theorem of calculus, $\norm{\partial_y c}_{L^\infty(\Torus)} \leq \norm{\partial_{yy} c_0}_{L^1(\Torus)}$, and hence the lemma follows. 
\end{proof}

\begin{lem}In the two-dimensional case, the following elliptic estimate holds:
\bel\label{elliptic estimate appendix}
||\na (c_{\neq})||_{L^\infty(\mathbb{T}^2)} \lesssim \norm{n_{\neq}}_{L^3(\Torus^2)}. 
\eel
\end{lem}
\begin{proof}
By Morrey's inequality, there holds for any $p > 2$, 
\begin{align*}
\norm{\grad c_{\neq}}_{L^\infty(\Torus^2)} \lesssim_p \norm{\grad c}_{L^p} + \norm{\grad^2 c}_{L^p}.  
\end{align*} 
The lemma follows from the Calderon-Zygmund inequality and the lack of low frequencies. 
\end{proof}

In the 3-dimensional case, we need the following lemmas.
\begin{lem}In the 3-dimensional case, the following mode by mode estimates are true:
\bel\label{mode by mode estimates 3D}\ba
||\na_y \widehat{c}_k||_{L^\infty(\rr^2)}\lesssim&||\widehat{n}_k||_{L^2(\rr^2)}^{\frac{1}{2}}||\na_y \widehat{n}_k||_{L^2(\rr^2)}^{\frac{1}{2}}\\
||\na_y \widehat{c}_k||_{L^\infty(\mathbb{T}^2)}\lesssim&||\widehat{n}_k||_{L^2(\mathbb{T}^2)}^{\frac{1}{2}}||\na_y \widehat{n}_k||_{L^2(\mathbb{T}^2)}^{\frac{1}{2}}.
\ea\eel
\end{lem}
\begin{proof}
From the Gagliardo-Nirenberg-Sobolev inequality, 
\begin{align*}
\norm{\grad c_{k}}_{L^\infty} \lesssim \norm{\grad c_{k}}_{L^2}^{1/2}\norm{\grad^3 c_{k}}_{L^2}^{1/2}, 
\end{align*}
from which the result follows. 
\end{proof}

Other than the lemma above, we need the following 3D elliptic estimates.  
\begin{lem}In the three-dimensional case, the following elliptic estimates are true:
\bel\label{elliptic estimate appendix 3D}\ba
||\na c_{\neq}||_{L^\infty(\mathbb{T}\times \rr^2)} \lesssim& ||n_{\neq}||_{L^4(\mathbb{T}\times \rr^2)}, \\ 
||\na c_{\neq}||_{L^\infty(\mathbb{T}^3)}^2\lesssim&  ||n_{\neq}||_{L^4(\mathbb{T}\times \rr^2)}, \\
||\na_y c_0||_{L^\infty(\rr^2)}\lesssim& ||n_0||_{L^1(\rr^2)}^{1/4}||n_0||_{L^3(\rr^2)}^{3/4},\\
||\na_y c_0||_{L^\infty(\mathbb{T}^2)}\lesssim& \norm{n_0-\overline{n}}_{L^3(\Torus^2)}.
\ea\eel
\end{lem}
\begin{proof} 
The first two inequalities follow from Morrey's inequality and the Calderon-Zygmund inequality, as above. 
Similarly, as does the last inequality. 
The third inequality follows from a standard argument: optimizing over the choice of $R$ we have, 
\begin{align*}
\abs{\grad_y c_0(y)} & \lesssim \abs{\int_{\abs{y-y'} \geq R} \frac{y-y'}{\abs{y-y'}^2} n_0(y') dy'} + \abs{\int_{\abs{y-y'} < R} \frac{y-y'}{\abs{y-y'}^2} n_0(y') dy'} \\ 
& \lesssim \frac{1}{R}\norm{n_0}_{L^1} + \norm{n_0}_{L^3} \left(\int_{\abs{y'} < R} \frac{1}{\abs{y'}^{3/2}} dy'\right)^{2/3} \\ 
& \lesssim \frac{1}{R}\norm{n_0}_{L^1} + \norm{n_0}_{L^3} R^{1/3} \\ %R^{4/3} = \norm{n}_1/\norm{n}_{3} 
& \lesssim \norm{n}_{L^1}^{1/4}\norm{n_0}^{3/4}. 
\end{align*}
\end{proof}

\subsubsection*{Acknowledgments}
The authors would like to thank Michele Coti Zelati for helpful discussions and Eitan Tadmor for helpful discussions and for suggesting this problem. 

%\vfill\eject
\bibliographystyle{abbrv}
\bibliography{nonlocal_eqns,JacobBib}

\end{document}